\documentclass[11pt]{article}

\usepackage{amsthm,amsmath,amsfonts,amssymb,bm}

\usepackage[utf8]{inputenc}

\usepackage[table,dvipsnames]{xcolor}
\usepackage{hhline}
\usepackage{caption}
\usepackage{enumerate}
\usepackage{hyperref}
\usepackage{bbm}
\RequirePackage{booktabs}

\RequirePackage[numbers]{natbib}

\usepackage[margin=3cm]{geometry}
\usepackage{graphicx}

\newcommand{\reff}[1]{(\ref{#1})}

\newcommand{\IW}{\mathbb{W}}

\newcommand{\IL}{\mathbb{L}}

\newcommand{\sH}{\mathcal{H}}
\newcommand{\sG}{\mathcal{G}}

\newcommand{\sP}{\mathcal{P}}
\newcommand{\sU}{\mathcal{U}}
\newcommand{\sV}{\mathcal{V}}
\newcommand{\sW}{\mathcal{W}}

\newcommand{\sM}{\mathcal{M}}

\newcommand{\IR}{\mathbb{R}}
\newcommand{\IC}{\mathbb{C}}
\newcommand{\IZ}{\mathbb{Z}}
\newcommand{\IN}{\mathbb{N}}
\newcommand{\IE}{\mathbb{E}}
\newcommand{\IP}{\mathbb{P}}

\newcommand{\Ii}{\mathbbm{1}}

\newcommand{\Var}{\mathbb V\mathrm{ar}}

\DeclareMathOperator*{\argmin}{\arg\min}

\newtheorem{thm}{Theorem}[section]
\newtheorem{lem}[thm]{Lemma}
\newtheorem{cor}[thm]{Corollary}
\newtheorem{prop}[thm]{Proposition}

\theoremstyle{definition}

\newtheorem{rem}[thm]{Remark}

\newtheorem{exs}[thm]{Examples}

\numberwithin{equation}{section}

\begin{document}

\title{Notes on Time Series}

\thispagestyle{empty}

\begin{center}
{\LARGE \bf Multiplicative deconvolution in survival analysis \\ under dependency}\\
\vspace{2ex}

{\large Sergio Brenner Miguel and Nathawut Phandoidaen}\\

\url{brennermiguel@math.uni-heidelberg.de}, \url{ phandoidaen@math.uni-heidelberg.de}\\

{\small Institut für angewandte Mathematik, Im Neuenheimer Feld 205, Universität Heidelberg}\\
\today
\end{center}

\begin{abstract}
 We study the non-parametric estimation of an unknown survival function $S$ with support on $\IR_+$ based on a sample with multiplicative measurement errors. 
  The proposed fully-data driven procedure is based on the estimation of the Mellin
  transform of the survival function $S$ and a regularisation of the inverse of
  the Mellin transform by a spectral cut-off. The upcoming bias-variance trade-off  is dealt with by a
  data-driven choice of the cut-off parameter. In order to discuss the bias term, we consider the
  \textit{Mellin-Sobolev spaces} which characterize the
  regularity of the unknown survival function $S$ through the decay of its Mellin
  transform. For the analysis of the variance term, we consider the i.i.d. case and incorporate dependent observations in form of Bernoulli shift processes and beta mixing sequences. Additionally,  we show minimax-optimality over \textit{Mellin-Sobolev spaces} of the spectral cut-off estimator. 
\end{abstract}

{\footnotesize
\begin{tabbing} 
\noindent \emph{Keywords:} \=  Survival analysis,
Non-parametric statistics, Multiplicative measurement errors, \\
Functional dependence measure, Adaptivity, Minimax theory \\
\noindent\emph{MSC 2010 subject classifications:}  Primary 62G05; secondary 62N02 , 62C20. 
\end{tabbing}}

\section{Introduction}
\label{sec_intro}
In this work we are interested in estimating the unknown survival function
 $S:\IR_+ \rightarrow \IR_+$ of
 a positive random variable $X$ given identically
 distributed copies of $Y=XU$ where $X$ and $U$ are
 independent of each other and $U$ has a known density $g:\IR_+
 \rightarrow \IR_+$. In this setting the density  $f_{Y}:\IR_+
 \rightarrow \IR_+$ of $Y$ is given by
 \begin{equation*}
 f_{Y}(y)=[f * g](y):= \int_{0}^{\infty} f(x)g(y/x) x^{- 1}d x\quad\forall  y\in\IR_+.
 \end{equation*}
 Here ``$*$" denotes multiplicative convolution. The estimation of
 $S$ using a sample $Y_1, \dots, Y_n$ from $f_{Y}$ is thus an
 inverse problem called
 multiplicative deconvolution.  We will allow for certain dependency structures on the sample $Y_1, \dots, Y_n$. More precisely, we assume that $X_1, \dots, X_n$ is a stationary process while the error terms $U_1,\dots, U_n$ will be independent and identically distributed (i.i.d.).\\
 As far as the multiplicative deconvolution model is concerned, the recent work of \cite{Brenner-MiguelComteJohannes2020} should be mentioned, which uses the Mellin transform to construct a density estimator under multiplicative measurement errors, and \cite{miguel2021anisotropic}, where they consider the multivariate case of density estimation. The model of multiplicative measurement errors was motivated in the work of \cite{BelomestnyGoldenshluger2020} as a generalization of several models, for instance the multiplicative censoring model or the stochastic volatility model. \\
 To the knowledge of the authors, the estimation for the survival function of a positive random variable for general multiplicative measurement errors has not been studied yet.\\
 The investigations done in \cite{Vardi1989} and \cite{VardiZhang1992} introduce and study
 \textit{multiplicative censoring} intensively, which corresponds to
 the particular multiplicative deconvolution problem with uniformly distributed multiplicative error $U$ on $[0,1]$. This
 model is often applied in survival analysis and was
 motivated in \cite{Van-EsKlaassenOudshoorn2000}. The estimation of
 the cumulative distribution function of $X$ is discussed in
 \cite{AsgharianWolfson2005} and
 \cite{VardiZhang1992}. Series expansion methods are
 discussed in \cite{AndersenHansen2001} treating the model as an
 inverse problem. The survival function estimation in a multiplicative
 censoring model is considered in \cite{BrunelComteGenon-Catalot2016}
 using a kernel estimator and a convolution power kernel
 estimator. Assuming an uniform error distribution on an interval
 $[1-\alpha, 1+\alpha]$ for $\alpha\in (0,1)$, \cite{ComteDion2016}
 analyzes a projection survival function estimator with respect to the Laguerre
 basis. On the other hand, \cite{BelomestnyComteGenon-Catalot2016} sheds light on the theory with
 beta-distributed error $U$.\\
 In the work of \cite{BelomestnyGoldenshluger2020}, the authors used the Mellin transform to construct a kernel estimator for the pointwise density estimation.  Moreover, they point out that the following widely used
 naive approach is a special case of their estimation
 strategy. Transforming the data by applying the logarithm the model
 $Y=XU$ writes as $\log(Y)=\log(X)+\log(U)$. In other words,
 multiplicative convolution becomes convolution for the
 $\log$-transformed data. As a consequence, the density of $\log(X)$
 is eventually estimated employing usual strategies for
 non-parametric deconvolution problems (cf. \cite{Meister2009}) and then transformed back to an estimator of $f$.
 However, it is difficult to interpret regularity conditions on the
 density of $\log(X)$. Furthermore, the analysis of a global risk of
 an estimator using this naive approach is challenging as
 \cite{ComteDion2016}
 pointed out. \\
 In this work, we extend the results of \cite{Brenner-MiguelComteJohannes2020} 
 for the estimation of the survival function.
 To do so, we introduce the Mellin transform for positive random variables and collect the necessary properties it.
 The key to the analysis of the multiplicative deconvolution problem is the multiplication
 theorem of the Mellin transform, which roughly states
 $\sM[f_Y]=\sM[f]\sM[g]$ for a density $f_Y=f
 *g$. Exploiting the
 multiplication theorem of the Mellin transform and applying a spectral cut-off regularisation of the inversion of the Mellin-transform we define a survival function estimator. We measure the accuracy of the estimator
 by introducing a global risk in terms of the
 $\IL^2$-norm. Making use of properties of the Mellin transform we
 characterize the underlying inverse problem and natural regularity
 conditions which borrow ideas from the inverse problems community
 (\cite{EnglHanke-BourgeoisNeubauer2000}).  The regularity conditions are expressed
 in the form of \textit{Mellin-Sobolev spaces} and their relations to
 the analytical properties of the survival function $S$ are discussed in more
 details. The proposed estimator, however, involves a tuning
 parameter which is selected by a data-driven method. We establish an
 oracle inequality for the fully-data driven spectral cut-off
 estimator under fairly mild assumptions on the error density
 $g$. 
 Moreover we show that uniformly over \textit{Mellin-Sobolev spaces} the proposed
 data-driven estimator is minimax-optimal. Precisely, we state both an upper
 bound for the mean  integrated squared error of the fully-data driven spectral
 cut-off  estimator and a general lower bound for estimating the
 density $f$  based on copies from $f_{ Y}=f *g$.\\
 
 Besides the discussion of i.i.d. samples, we also examine the estimator's behavior when certain dependency structures are present, similar to \cite{ComteDedeckerTaupin2008} who considered density estimation for general ARCH models, by using the log-transformed data approach.  We will briefly introduce the two central concepts of dependency we base our work on. A classical structure in literature is given by absolutely regular mixing coefficients which are also called $\beta$-mixing coefficients.
 
\paragraph{Absolutely regular mixing ($\beta$-mixing)}
First, let us consider two sigma fields $\sU, \sV$ over some probability space $\Omega$ and define the quantity 
\[
    \beta^{\textit{mix}}(\sU,\sV) := \frac{1}{2}\sup \sum_{(p,q) \in P \times Q}|\IP(U_p \cap V_q) - \IP(U_p) \IP(V_q)|
\] where the supremum is taken over all finite partitions $(U_p)_{p\in P}$, $(V_q)_{q\in Q}$ of $\Omega$ such that $(U_p)_{p\in P} \subset \sU$, $(V_q)_{q\in Q} \subseteq \sV$. Now, let $(\beta(k))_{k \in \IN_0}$ be a sequence of real valued numbers defined by
\begin{equation}
    \beta(k) := \beta(X_0,X_k) := \beta^{\textit{mix}}(\sigma(X_0), \sigma(X_k)),\label{definition_beta_mixing}
\end{equation} for $\sigma$-fields generated by $X_0$ and $X_k$, respectively.
Then, the process is said to be $\beta$-mixing if for the corresponding coefficients $\beta(k) \to 0$ as $k \to \infty$.
Graphically, $\beta(k)$, $k\in\IN_0$, measures the dependence between $\sigma(X_0)$ and $\sigma(X_k)$ and decays to $0$ for $k \to \infty$ if $\sigma(X_0)$ contains no information about $X_k$ for large $k$. We refer to \cite[Section 1.3]{rio2013} or \cite{rio1995invariance} for a more detailed introduction. There are several results available which state that many linear processes of current interest, such as GARCH or ARMA processes, have absolutely summable $\beta(k)$, cf. \cite{bradley2005}, \cite{fryzlewicz2011} or \cite{doukhan_mixingbook}.

\paragraph{Functional dependence measure}
In recent years it has become popular to study dependency behavior by the so called functional dependence measure (cf. \cite{wu2005anotherlook}). In this context it is assumed that the given process has a representation as a Bernoulli shift process. More precisely, we let $X_j$, $j = 1,...,n$, be a one dimensional process of the form
\begin{equation}
    X_j = J_{j,n}(\sG_j),\label{def:bernoulli-shift}
\end{equation}
where $\sG_j = \sigma(\varepsilon_j,\varepsilon_{j-1},...)$ is the sigma-algebra generated by $\varepsilon_j$, $j \in\IZ$, a sequence of i.i.d. random variables in $\IR$, and some measurable function $J_{j,n}:(\IR)^{\IN_0}\to \IR$, $j=1,...,n$, $n\in\IN$. For a real-valued random variable $W$ and some $\nu > 0$, we define $\|W\|_{\nu} := \IE[|W|^{\nu}]^{1/\nu}$. 
If $\varepsilon_k^{*}$ is an independent copy of $\varepsilon_k$, independent of $\varepsilon_j, j\in\IZ$, we define $\sG_j^{*(j-k)} := (\varepsilon_j,...,\varepsilon_{j-k+1},\varepsilon_{j-k}^{*},\varepsilon_{j-k-1},...)$ and $X_j^{*(j-k)} := J_{j,n}(\sG_{j}^{*(j-k)})$. Then, the functional dependence measure of $X_j$ is given by
\begin{equation}
    \delta_{\nu}^{X}(k) = \big\|X_{j} - X_{j}^{*(j-k)}\big\|_\nu.\label{def:functional_dependence_measure}
\end{equation}
Graphically, $\delta_{\nu}^{X}$ measures the impact of $\varepsilon_0$ in $X_k$ with respect to the appropriate norm. Although representation \reff{def:bernoulli-shift} appears to be rather restrictive, it does cover a large variety of  processes. In \cite{borkar1993} it was motivated that the set of all processes of the form $X_j = J_{j,n}(\varepsilon_j,\varepsilon_{j-1},...)$ should be equal to the set of all stationary and ergodic processes.

The merit of this approach can be highlighted by the ease of computation for a broad range of processes, in contrast to mixing coefficients.

So, we can summarize the novel contributions in survival function estimation with multiplicative errors via Mellin transforms by providing in the setting of i.i.d. and dependent observations,
\begin{itemize}
    \item an upper bound for the spectral cut-off estimator's global risk and the rates in the Mellin-Sobolev space (cf. Theorem \ref{thm:upp_bound} and Corollary)
    \item minimax-optimality specifically for i.i.d. data (cf. Theorem \ref{theorem:lower_bound}) as well as
    \item an upper bound and rates when a data driven method is applied for selecting an appropriate tuning parameter (cf. Theorem \ref{thm:upper_bound_adap} and Corollary).
\end{itemize}

Our discourse is organized as follows. In Section \ref{sec_minimax} we revise the Mellin transform, give examples and summarize its main properties. Thereafter, we establish our estimation strategy for the survival function. We provide oracle inequalities for independent and dependent data and distill with respect to the MISE an upper bound with parametric as well as non-parametric rates in an appropriate Mellin-Sobolev space.
We conclude the section by deriving the estimator's minimax optimality in the i.i.d. case. Since our theory depends on an spectral cut-off parameter, we propose in Section \ref{sec_datadriven} a data-driven method based on a penalized contrast approach to select a suitable tuning parameter. As before, we state an oracle inequality and derive an upper bound, accordingly. To illustrate our results, we showcase numerical studies in Section \ref{sec_numerical}. The proofs can be found in the Appendix, i.e. Section \ref{sec_append}.

\section{Minimax}
\label{sec_minimax}

In this section we introduce the Mellin transform and collect its relevant properties for our theory.  
We define for any weight function
$\omega:\IR \rightarrow \IR_+$ the corresponding
weighted norm  by $\|h\|_{\omega}^2 := \int_0^{\infty}
|h(x)|^2\omega(x)dx $ for a measurable, complex-valued function $h$. Denote by
$\IL^2(\IR_+,\omega)$ the set of all measurable, complex-valued functions with
finite $\|\, .\,\|_{\omega}$-norm and by $\langle h_1, h_2
\rangle_{\omega} := \int_0^{\infty}  h_1(x)
\overline{h_2}(x)\omega(x)dx$ for $h_1, h_2\in \IL^2(\IR_+,\omega)$
the corresponding weighted scalar product. Similarly, define $\IL^2(\IR):=\{ h:\IR \rightarrow \IC\, \text{ measurable }: \|h\|_{\IR}^2:= \int_{-\infty}^{\infty} h(t)\overline{h(t)} dt <\infty \}$ and $\IL^1(\Omega,\omega):=\{h: \Omega \rightarrow \IC: \|h\|_{\IL^1(\Omega,\omega)}:= \int_{\Omega} |h(x)|\omega(x)dx < \infty \}.$

In the upcoming theory, we need to ensure that the survival function $S$ of the sample $X_1, \dots, X_n$ is square-integrable. Furthermore, to define the estimator, we additionally need the square integrablility of the empirical survival function $\widehat S_X$ which is defined by
\begin{align}\label{eq:emp:surv:x}
    \widehat S_X(x):= n^{-1} \sum_{j=1}^n \Ii_{(0,X_i)}(x)
\end{align}
for any $x\in \IR_+$. The following proposition shows that we can derive the square-integrablilty condition for both functions by a moment condition.

\begin{prop}\label{prop:surv}
 Let $\IE(X^{1/2})< \infty$. Then, $S \in \IL^1(\IR_+,x^{-1/2})\cap \IL^2(\IR, x^{0})$. If additionally $\IE(X_1)<\infty$ then $\widehat S_X \in \IL^1(\IR_+,x^{-1/2})\cap \IL^2(\IR, x^{0})$ almost surely.
\end{prop}
The proof of Proposition \ref{prop:surv} can be found in the Appendix.
We define an estimator of $S$ based on the contaminated data $Y_1, \dots, Y_n$ and use the rich theory of Mellin transforms. Before we summarize the main properties of the Mellin transform let us shortly study the upcoming multiplicative convolution.
\paragraph{Multiplicative Convolution}
In the introduction we already mentioned that the density $f_Y$ of $Y_1$ can be written as the multiplicative convolution of the densities $f$ and $g$. We will now define this convolution in a more general setting. Let $c\in \IR$. For two functions $h_1,h_2\in \IL^1(\IR_+, x^{c-1})$ we define the multiplicative convolution $h_1*h_2$ of $h_1$ and $h_2$ by
\begin{align}\label{eq:mult:con}
    (h_1*h_2)(y):=\int_0^{\infty} h_1(y/x) h_2(x) x^{-1} dx, \quad y\in \IR.
\end{align}
In fact, it is possible to show that the function $h_1*h_2$ is well-defined, $h_1*h_2=h_2*h_1$ and $h_1*h_2 \in \IL^1(\IR,x^{c-1})$, compare \cite{miguel2021anisotropic}. It is worth pointing out, that the definition of the multiplicative convolution in equation \reff{eq:mult:con} is independent of the model parameter $c\in \IR$. We also know that for densities $h_1,h_2$; $h_1,h_2\in \IL^1(\IR_+,x^0)$. If additionally $h_1\in \IL^2(\IR_+, x^{2c-1})$  then $h_1*h_2 \in \IL^2(\IR_+,x^{2c-1})$. 

\paragraph{Mellin transform}
We will now collect the main properties of the Mellin transform which will be used in the upcoming theory. Proof sketches of these properties can be found in \cite{miguel2021anisotropic}. Let $h_1\in \IL^1(\IR,x^{c-1})$. Then, we define the Mellin transform of $h_1$ at the development point $c\in \IR$ as the function $\sM_c[h]:\IR\rightarrow \IC$ with
\begin{align}
    \sM_c[h_1](t):= \int_0^{\infty} x^{c-1+it} h_1(x)dx, \quad t\in \IR.
\end{align}
One key property of the Mellin transform, which makes it so appealing for the use of multiplicative deconvolution, is the so-called convolution theorem, that is, for $h_1, h_2\in \IL^1(\IR_+,x^{c-1})$,
\begin{align}
    \sM_c[h_1*h_2](t)=\sM_c[h_1](t) \sM_c[h_2](t), \quad t\in \IR.
\end{align}
Addtionally, for the estimation of the survival function the following property is used. Let $h\in \IL^1(\IR_+,x^{c-1})$ be a density and $S_h:=\IR_+\rightarrow \IR_+, y \mapsto \int_y^{\infty} h(x)dx$ its corresponding survival function. Then for any $c>0$, $S_h\in \IL^1(\IR_+, x^{c-1})$ if and only if $h\in \IL^1(\IR_+, x^{c})$. Furthermore, for any $t\in \IR$,
\begin{align*}
    \sM_c[S_h](t)=(c+it)^{-1}\sM_{c+1}[h](t).
\end{align*}
Let us now define the Mellin transform of a  square-integrable function, that is, for $h_1\in \IL^2(\IR_+, x^{2c-1})$ we make use of the definition of the Fourier-Plancherel transform. To do so, let $\varphi:\IR \rightarrow \IR_+, x\mapsto \exp(-2\pi x)$ and $\varphi^{-1}: \IR_+ \rightarrow \IR$ be its inverse. Then, as diffeomorphisms, $\varphi, \varphi^{-1}$ map  Lebesgue null sets on Lebesgue null sets. Thus the isomorphism $\Phi_c:\IL^2(\IR_+,x^{2c-1}) \rightarrow \IL^2(\IR), h\mapsto \varphi^c \cdot(h\circ \varphi)$ is well-defined. Moreover, let $\Phi^{-1}_c: \IL^2(\IR) \rightarrow \IL^2(\IR_+,x^{2c-1})$ its inverse. Then for $h\in \IL^2(\IR_+,x^{2c-1})$ we define the Mellin transform of $h$ developed in $c\in \IR$ by
\begin{align*}\label{eq:mel;def}
	\sM_c[h](t):= (2\pi) \mathcal F[\Phi_c[h]](t) \quad \text{for any } t\in \IR,
\end{align*} 
  where $\mathcal F: \IL^2(\IR)\rightarrow \IL^2(\IR), H\mapsto (t\mapsto \mathcal F[H](t):=\lim_{k\rightarrow \infty}\int_{-k}^k \exp(-2\pi i t x) H(x) dt)$ is the Plancherel-Fourier transform. Due to this definition several properties of the Mellin transform can be deduced from the well-known theory of Fourier transforms. In the case that $h\in \IL^1(\IR_+,x^{c-1}) \cap \IL^2(\IR_+,x^{2c-1})$ we have 

\begin{align}
	\sM_c[h](t) =\int_0^{\infty} x^{c-1+it} h(x)dx \quad \text{for any } t\in \IR
\end{align}
which coincides with the usual notion of Mellin transforms as considered in \cite{ParisKaminski2001}. 

Now, due to the construction of the operator $\mathcal M_c: \IL^2(\IR_+,x^{2c-1}) \rightarrow \IL^2(\IR)$ it can easily be seen that it is an isomorphism. We denote by $\mathcal M_c^{-1}: \IL^2(\IR) \rightarrow \IL^2(\IR_+,x^{2c-1})$ its inverse. If additionally to $H\in \IL^2(\IR)$, $H\in \IL^1(\IR)$, we can express the inverse Mellin transform explicitly through
\begin{align}\label{eq:Mel:inv}
\sM_{c}^{-1}[H](x)= \frac{1}{2\pi } \int_{-\infty}^{\infty} x^{-c-it} H(t) dt, \quad \text{ for any } x\in \IR_+.
\end{align} 
Furthermore, we can directly show that a Plancherel-type equation holds for the Mellin transform, that is for all $h_1, h_2 \in \IL(\IR_+,x^{2c-1})$, 
\begin{align}\label{eq:Mel:plan}
    \hspace*{-0.5cm}\langle h_1, h_2 \rangle_{x^{2c-1}} = (2\pi)^{-1} \langle \sM_c[h_1], \sM_c[h_2] \rangle_{\IR} \quad \text{ whence } \quad \| h_1\|_{x^{2c-1}}^2=(2\pi)^{-1} \|\sM_c[h]\|_{\IR}^2.
\end{align} 
Let us now construct our estimator based on the above properties.

\paragraph{Estimation strategy}
We consider the case $c=1/2$ where the weighted $\IL^2$-norm is the usual unweighted $\IL^2$-norm. 
Assuming now that $\IE(X_1^{1/2})< \infty$ we have for $k\in \IR_+$ that  $\sM_{1/2}[S]\Ii_{[-k,k]} \in \IL^1(\IR) \cap \IL^2(\IR)$ and thus
\begin{align}\label{eq:S_k_1}
    S_k(x):= \sM_{1/2}^{-1}[\sM_{1/2}[S]\Ii_{[-k,k]}](x)= \frac{1}{2\pi} \int_{-k}^k x^{-1/2-it} \sM_{1/2}[S](t)dt, \quad x\in \IR_+,
\end{align}
is an approximation of $S$ in the $\IL^2(\IR_+,x^0)$ sense, that is, $\|S_k-S\|^2\rightarrow0$ for $k\rightarrow\infty.$ 
Now, applying the property of the Mellin transform for survival functions, we know that $\sM_{1/2}[S](t)=(1/2+it)^{-1}\sM_{3/2}[f](t)$ for all $t\in \IR$. Assuming that $\IE(U^{1/2})<\infty$, and thus $\IE(Y^{1/2})<\infty$, we get $f_Y, g\in \IL^1(\IR, x^{1/2})$. We can deduce from the convolution theorem that $\sM_{3/2}[f_Y]=\sM_{3/2}[f] \sM_{3/2}[g]$. Under the mild assumption that forall $t\in \IR, \sM_{3/2}[g](t)\neq 0$ we can rewrite equation \reff{eq:S_k_1} in the following form
\begin{align}\label{eq:S_k_2}
    S_k(x)= \frac{1}{2\pi} \int_{-k}^k x^{-1/2-it} \frac{\sM_{3/2}[f_Y](t)}{(1/2+it)\sM_{3/2}[g](t)} dt.
\end{align}
To derive an estimator from equation \reff{eq:S_k_2} we use the empirical Mellin transform given by $\widehat{\sM}(t):=n^{-1}\sum_{j=1}^n Y_j^{1/2+it}$ as an unbiased estimator of $\sM_{3/2}[f_Y](t)$ for all $t\in \IR$. Keeping in mind that $|\widehat{\sM}(t)| \leq |\widehat{\sM}(0)|<\infty$ almost-surely, it is sufficient to assume that $\int_{-k}^k |(1/2+it)\sM_{3/2}[g](t)|^{-2}dt<\infty$ for all $k\in \IR_+$ to ensure the well-definedness of the spectral cut-off estimator
\begin{align}\label{eq:estim:def}
    \widehat S_k(x):= \frac{1}{2\pi} \int_{-k}^k x^{-1/2-it} \frac{\widehat{\sM}(t)}{(1/2+it)\sM_{3/2}[g](t)} dt, \quad k, x\in \IR_+.
\end{align}
Up to now, we had two minor conditions on the Mellin transform of the error density $g$ which we want to collect in the following assumption,
\begin{align*}
    \forall \,t \in \IR: \sM_{3/2}[g](t)\neq 0 \quad \text{ and } \quad \forall\, k\in \IR_+:\int_{-k}^k |(1/2+it)\sM_{3/2}[g](t)|^{-1} dt <\infty \tag{\textbf{[G0]}}.
\end{align*}
The following proposition shows that the proposed estimator is consistent for a suitable choice of a cut-off parameter and under certain assumptions on the dependency structure of $X_1, \dots, X_n$.
It is worth stressing out, that for $t\in \IR$ the estimator $(1/2+it)^{-1} \widehat \sM_X(t):=(1/2+it)^{-1} n^{-1}\sum_{j=1}^n X_j^{1/2+it}$ is an unbiased estimator of $\sM_{1/2}[S](t)$. Furthermore, there is a special link between the empirical survival function and the estimator $(1/2+it)^{-1/2} \widehat \sM(t)$. In fact, \reff{prop:surv} ensures that $\widehat S_X\in \IL^1(\IR_0,x^{-1/2})$ almost surely which implies the existence of  Mellin transform of $\widehat S_X$ almost surely. From that, it can easily be shown that 
\begin{align}\label{eq:mel:surv}
    \sM_{1/2}[\widehat S_X](t)= (1/2+it)^{-1}\widehat \sM_X(t)
\end{align}
for all $t\in \IR$. 
\begin{thm}\label{thm:upp_bound}
Assume that $\IE(Y)< \infty$ and that \textbf{[G0]} holds. Then for any $k\in \IR_+$,
\begin{align*}
    \IE(\|\widehat S_k-S\|^2) \leq \|S-S_k\|^2 + \IE(Y_1)\frac{\Delta_g(k)}{n} + \frac{1}{2\pi} \int_{-k}^k \Var(\sM_{1/2}[\widehat S_X](t)) dt,
\end{align*}
where $\Delta_g(k)= (2\pi)^{-1}\int_{-k}^k |(1/2+it)\sM_{3/2}[g](t)|^{-2}dt$.\\
If $(k_n)_{n\in \IN}$ is chosen such that $k_n \rightarrow \infty$ for $n\rightarrow \infty$, $\frac{1}{2\pi} \int_{-k_n}^{k_n} \Var(\sM_{1/2}[\widehat S_X](t)) dt \rightarrow 0$ and $\Delta_g(k_n)n^{-1}\rightarrow0$, the consistency of $\widehat S_{k_n}$, that is,
\begin{align*}
    \IE(\|\widehat S_{k_n}-S\|^2) \rightarrow 0, \qquad\text{ } n\rightarrow \infty
\end{align*} is implied.
\end{thm}

\begin{cor}\label{cor:consis_dep}
Under the assumptions of Theorem \reff{thm:upp_bound} we have
\begin{enumerate}
\item[\textbf{(I)}] for independent observations $X_1, \dots, X_n$,
\begin{align*}
    \IE(\|\widehat S_k-S\|^2) \leq \|S-S_k\|^2 + \IE(Y_1)\frac{\Delta_g(k)}{n} + \frac{\IE(X_1)}{n} ;
\end{align*}
\item[\textbf{(B)}] for $\beta$-mixing observations $X_1, \dots, X_n$ under the additional assumption that $\IE(X_1b(X_1))<\infty$,
\begin{align*}
    \IE(\|\widehat S_k-S\|^2) \leq \|S-S_k\|^2 + \IE(Y_1)\frac{\Delta_g(k)}{n} + \frac{c\IE(X_1b(X_1))}{n},
\end{align*}
where $c>0$ is a positive numerical constant;
\item[\textbf{(F)}] and for Bernoulli-shift processes \reff{def:bernoulli-shift} under the dependency measure \reff{def:functional_dependence_measure} provided that additionally $\sum_{j=1}^\infty\delta_1^X(j)^{1/2} < \infty$,
\begin{align*}
    \IE(\|\widehat S_k-S\|^2) \leq \|S-S_k\|^2 + \IE(Y_1)\frac{\Delta_g(k)}{n} + \frac{c\log(k)}{n}\Big(\sum_{j=1}^\infty\delta_1^X(j)^{1/2}\Big)^2
\end{align*}
where $c>0$ is a numerical positive constant.
\end{enumerate}
\end{cor}

As we can see, the first term, the so-called bias term, in the upper boundary in Theorem \ref{thm:upp_bound} is monotonically decreasing in $k\in \IR_+$ while the second and the last term are monotonically increasing in $k\in \IR_+$. In fact, the second and the third term are a decomposition of $\IE(\|\widehat S_k-S_k\|^2)$, the so-called variance term. While Corollary \ref{cor:consis_dep} indicates when the third term is bounded, respectively growing with $\log(k)$, the general assumptions on the error densities allow us not to determine the exact growth of the second term. For a more sophisticated analysis of the variance's grwoth, we need to consider more specific assumptions on the error density $g$.\\
More precisely, the growth of $\Delta_g$ is determined by the decay of the Mellin transform of $g$. In this work, we mainly focus on the case of \textit{smooth error densities}, that is, there exists $c,C, \gamma\in \IR_+$ such that
\begin{align*}
\tag{\textbf{[G1]}}  c(1+t^2)^{-\gamma/2} \leq |\sM_c[g]( t)| \leq C (1+t^2)^{-\gamma/2} \text{ for } t\in \IR.
\end{align*} 
This assumption on the error density is typical in context of additive deconvolution problems (compare \cite{Fan1991}) and was also considered in the works of \cite{BelomestnyGoldenshluger2020} and \cite{Brenner-MiguelComteJohannes2020}.

We present some Mellin transform of commonly used distribution families.
\begin{exs}\label{exs:mel:well}
		\begin{enumerate}
			\item \textit{Beta distribution}: Consider the family $(g_b)_{b\in \IN}$, $g_b(x):= \Ii_{(0,1)}(x) b(1-x)^{b-1}$ for $b\in \IN$ and $x\in \IR_+$. Obviously, we see that $\mathcal M_c[g_b]$ is well-defined for $c>0$ and
		\begin{align*}
		\sM_c[g_b](t) =\prod_{j=1}^{b} \frac{j}{c-1+j+it}, \quad t\in \IR.
		\end{align*} 
			\item \textit{Scaled Log-Gamma distribution}: Consider the family $(g_{\mu, a, \lambda})_{(\mu,a,\lambda) \in \IR\times \IR_+\times \IR_+}$ where for $a,\lambda, x \in \IR_+$ and $\mu \in \IR$ we have $g_{\mu, a, \lambda}(x)=\frac{\exp(\lambda \mu)}{\Gamma(a)} x^{-\lambda-1} (\log(x)-\mu)^{a-1}\Ii_{(e^{\mu}, \infty)}(x)$ . Then for $c<\lambda+1$,
		\begin{align*}
		    \sM_c[g_{\mu, a, \lambda}](t)= \exp(\mu(c-1+it)) (\lambda-c+1-it)^{-a}, \quad t\in \IR.
		\end{align*}
		If $a=1$ then $g_{\mu,1,\lambda}$ is the density of a Pareto distribution with parameter $e^{\mu}$ and $\lambda$. If $\mu=0$ we have that $g_{0, a, \lambda}$ is the density of a Log-Gamma distribution.
		\item \textit{Gamma distribution}: Consider the family $(g_d)_{d\in \IR_+}$,$g_d(x) = \frac{x^{d-1}}{\Gamma(d)} \exp(-x) \Ii_{\IR^+}(x)$ for $d,x\in \IR_+$. Obviously, we see that $\mathcal M_c[g_d]$ is well-defined for $c>-d+1$ and
		\begin{align*}
		\mathcal M_c[g_d](t)= \frac{\Gamma(c+d-1+it)}{\Gamma(d)}, \quad t\in \IR.
		\end{align*}
		\item \textit{Weibull distribution:} Consider the family $(g_m)_{m\in \IR_+}$, $g_m(x) = m x^{m-1} \exp(-x^m) \Ii_{\IR_+}(x)$ for $m,x\in \IR_+$. Obviously, we see that $\mathcal M_c[g_m]$ is well-defined for $c>-m+1$ and
		\begin{align*}
		\mathcal M_c[g_m](t)= \frac{(c-1+it)}{m}\Gamma(\frac{c-1+it}{m}), \quad t\in \IR.
		\end{align*}
		\item \textit{Log-normal distribution:}  Consider the family $(g_{\mu,\lambda})_{(\mu, \lambda) \in \IR\times \IR_+}$ where $g_{\mu,\lambda}$ for $\lambda, x \in \IR_+$ and $\mu \in \IR$ is given by $g_{\mu,\lambda}(x)=\frac{1}{\sqrt{2\pi}\lambda x} \exp(-(\log(x)-\mu)^2/2\lambda^2)\Ii_{\IR_+}(x)$. We see that $\sM_c[g_{\mu,\lambda}]$ is well-defined for any $c\in \IR$ and
			\begin{align*}
			\sM_c[g_{\mu,\lambda}](t)= \exp(\mu(c-1+it))\exp\left(\frac{\lambda^2(c-1+it)^2}{2}\right), \quad t\in \IR.
			\end{align*}
		\end{enumerate}
\end{exs}

In context of smooth error densities, we therefore have  the following with regards to \textbf{[G1]}.

\begin{exs}[Examples \ref{exs:mel:well}, continued]
\begin{enumerate}
    \item[(i)] \textit{Beta distribution:} For $b\in \IN$ and $t\in \IR$ we have already seen that $\sM_{3/2}[g_b](t)= \prod_{j=1}^{b} \frac{j}{1/2+j+it}$. Therefore, there exists $c_g, C_g>0$ such that
    \begin{align*}
        c_g(1+t^2)^{-b/2} \leq |\sM_{3/2}[g_b](t)| \leq C_g (1+t^2)^{-b/2}.
    \end{align*}
    \item[(ii)] \textit{Scaled Log-Gamma distribution:} For $\lambda >1/2$, $a\in \IR_+$ and $\mu, t\in \IR$ we have already seen that $\sM_{3/2}[g_{\mu,a,\lambda}](t)= \exp(\mu(1/2+it))(\lambda-1/2-it)^{-a}$. Therefore, there exists $c_g, C_g>0$ such that
    \begin{align*}
        c_g(1+t^2)^{-a/2} \leq |\sM_{3/2}[g_{\mu,a,\lambda}](t)| \leq C_g (1+t^2)^{-a/2}.
    \end{align*}
\end{enumerate}
\end{exs}
We would like to mention, that for small values of $\gamma$ in \textbf{[G1]}, it is possible to choose $k$ independent of the decay of the bias term, such that the risk in Theorem \ref{thm:upp_bound} is of order $n^{-1}$, respectively $\log(n)n^{-1}$. These cases are covered in the following section.

\paragraph{Parametric rate}
The special case $\gamma \leq 1/2$ in \textbf{[G1]} allows for a choice of the parameter $k\in \IR_+$ which leads to a parametric rate up to a log-term. It is worth pointing out, that this choice can be done independently of the precise decay of the bias term $\|S-S_k\|^2$. This can be accomplished by the naive bound
\begin{align*}
    \|S-S_k\|^2 = \frac{1}{\pi} \int_{k}^{\infty} |\sM_{1/2}[S](t)|^2 dt \leq k^{-1} \IE(X_1^{1/2})^{2}
\end{align*}
where we exploit $\sM_{1/2}[S](t)=(1/2+it)^{-1} \sM_{3/2}[f](t)$ and the bound $|\sM_{3/2}[f](t)|\leq \IE(X_1^{1/2})$. The different cases are collected in the following Proposition whose proof is omitted.

\begin{prop}
Let $\IE(Y)<\infty$ and let \textbf{[G1]} hold for $\gamma\leq 1/2$. Then,
\begin{align*}
    \Delta_g(k) \leq \begin{cases}  C(g) &, \gamma < 1/2;\\
    C(g)\log(k)& , \gamma = 1/2.
    \end{cases}
\end{align*}
For $\gamma=1/2$, choosing $k=n$ leads in all three cases \textbf{(I)},\textbf{(B)} and \textbf{(F)} to a parametric rate up to a log-term, that is,
\begin{align*}
    \IE(\|\widehat S_n-S\|^2) \leq C(f,g) \frac{\log(n)}{n}
\end{align*}
where $C(f,g)$ is dependent on $\IE(X_1),\IE(U_1)$, the constants in \textbf{[G1]} and the dependency structure.\\
If $\gamma<1/2$ in the cases \textbf{(I)} and \textbf{(B)}, choosing $k=\infty$ leads to a parametric rate, that is
\begin{align*}
    \IE(\|\widehat S_{\infty}- S\|^2) \leq \frac{C(f,g)}{n}
\end{align*}
where $C(f,g)$ depends of $\IE(X_1), \IE(Y_1)$ and the constants in \textbf{[G1]} and the dependency structure.
If $\gamma<1/2$ and we are in the case of \textbf{(F)} the third summand in Corollary \ref{cor:consis_dep} dominates the second, which leads to no improvement in the rate.
\end{prop}

\paragraph{Non-parametric rate} Now let us consider the case where $\gamma>1/2$ which forces $\Delta_g(k)$ to be polynomial increasing. In fact, 
under \textbf{[G1]} we see that $c_g k^{2\gamma-1}\leq \Delta_g( k)\leq C_g k^{2\gamma-1}$ for every $\ k \in \IR_+$. \\
To control the bias-term we will introduce regularity spaces characterized by the decay of the Mellin transform in analogy to the usual considered Sobolev spaces for deconvoultion problems.
Let us for $s\in \IR_+$ define the \textit{Mellin-Sobolev space} by
\begin{align}\label{eq:ani:mell:sob}
\sW^{ s}_{1/2}(\IR_+):= \{ h\in \IL^2(\IR_+,  x^{0}): |h|_{s}^2:= \|(1+t)^s\sM_{1/2}[h]\|_{\IR}^2<\infty\} 
\end{align}
and the corresponding ellipsoids with $L\in \IR$ by 
$
    \sW^{ s}_{1/2}(L):=\{ h\in \sW^{s}_{1/2}(\IR_+): |h|_{ s}^2 \leq L\}
$.
For $f\in \sW^{s}_{1/2}(L)$ we deduce that $\|S-S_{k}\|^2 \leq L k^{-2s} $.
Setting 
\begin{align*}
    \IW^{s}_{1/2}(L):=\{S \in \sW^{s}_{1/2}(L):  \text{S survival function}, \Var_S(\sM_{1/2}[\widehat S_X](t))\leq L (1+|t|)^{-1} n^{-1} \text{ for any } t\in \IR \}
\end{align*} and the previous discussion leads to the following statement whose proof is omitted.
\begin{prop}\label{prop:upper:minimax}
	Let $\IE(U)<\infty$. Then under the assumptions \textbf{[G0]} and \textbf{[G1]},
	\begin{align*}
	\sup_{S\in \IW^{s}_{1/2}(L)} \IE(\|S- \widehat S_{ k_n}\|)^2 \leq C(L,g , s) n^{-2s/(2s+2\gamma-1)} 
	\end{align*}
	for the choice $k_n:=n^{1/(2s+2\gamma-1)}$.
\end{prop}
Again let us consider the three different cases of dependency considered in Corollary \ref{cor:consis_dep}. Then we get as a direct consequence of Proposition \ref{prop:upper:minimax} and Corollary \ref{cor:consis_dep} the following Corollary.
\begin{cor}\label{cor:upper:minimax}
    The assumption that $\Var(\sM_{1/2}[\widehat S_X](t))\leq L (1+|t|)^{-1} n^{-1} \text{ for any } t\in \IR$ in Proposition \ref{prop:upper:minimax} can be replaced in the three different dependency cases by
    \begin{enumerate}
        \item[\textbf{(I)}] $\IE(X_1) \leq L$
        \item[\textbf{(B)}] $\IE(X_1b(X_1))\leq L$
        \item[\textbf{(F)}] $\sum_{k=1}^\infty \delta_1^X(k)^{1/2} \leq L^{1/2}$
        
    \end{enumerate}
\end{cor}
We will now show, that the rates presented in Corollary \ref{cor:upper:minimax} are optimal in the sense, that there exists no estimator based on the i.i.d. sample $Y_1,\dots, Y_n$ that can reach uniformly over $\IW_{1/2}^s(L)$ a better rate. This implies that the estimator $\widehat S_{k_n}$ presented in \ref{prop:upper:minimax} is minimax-optimal. \\
For technical reason we need an additional assumption on the error density $g$. Let us assume that the support of $g$ is bounded, that is there exists an $d>0$ such that 
 for all $x \geq d, g(x)=0$. For the sake of simplicity we assume that $d=1$. Further, let there be constants $c, C, \gamma\in \IR_+$ such that
\begin{align*}
\tag{\textbf{[G2]}}  c(1+t^2)^{-\gamma/2} \leq |\sM_{1/2}[g](t)| \leq C (1+t^2)^{-\gamma/2} \text{ for } t\in \IR.
\end{align*} 

With this additional assumption we can show the following theorem where its proof can be found in Appendix \ref{a:mt}.

\begin{thm}\label{theorem:lower_bound}
	Let $s,\gamma\in \IN$ and assume that \textbf{[G1]} and \textbf{[G2]} hold. Then there exist
	constants $c_{g},L_{s,g}>0$ such that for all
	$L\geq L_{s,g}$, $n\in \IN$, and for any estimator $\widehat S$ of $S$ based
	on an i.i.d. sample $Y_1,\dots, Y_n$, 
		\begin{align*}
		\sup_{S\in\IW^{s}_{1/2}(L)}\IE(\|\widehat S-S\|^2) \geq c_{g}   n^{-2s/(2s+2\gamma-1)}.
		\end{align*}
\end{thm}
For the multiplicative censoring model, that is, $U$ is uniformly-distributed on $[0,1]$ the assumptions \textbf{[G1]} and \textbf{[G2]} hold true. \\
Nevertheless, the rate presented in Proposition \ref{prop:upper:minimax} is very pessimistic, meaning that we can find examples of $f\in \IW_{1/2}^s(L)$ where the bias decays faster than $Lk^{-2s}$. These examples are considered in the next section.

\paragraph{Faster rates}

Let us revisit the families (iii)-(v) in Example \ref{exs:mel:well}.

\begin{exs}[Example \ref{exs:mel:well}, continued]
By application of the Stirling formula,
\begin{enumerate}
    \item the \textit{Gamma distribution} $f(x)= \frac{x^{d-1}}{\Gamma(d)} \exp(-x)\Ii_{\IR_+}(x)$, $d,x\in \IR_+$ delivers
    \begin{align*}
        |\sM_{1/2}[S](t)|=|(1/2+it)|^{-1} |\sM_{3/2}[f](t)| \leq C_d |t|^{d-1} \exp(-\pi|t|/2), \text{ for } |t|\geq 2;
    \end{align*}
    \item the \textit{Weibull distribution} $f(x)= m x^{m-1}\exp(-x^m)\Ii_{\IR_+}(x)$, $m,x\in \IR_+$, delivers
    \begin{align*}
        |\sM_{1/2}[S](t)|=|(1/2+it)|^{-1} |\sM_{3/2}[f](t)| \leq C_m |t|^{(1-m)/2m} \exp(-\pi|t|/(2m)), \text{ for } |t|\geq 2;
    \end{align*}
    \item the \textit{Log-normal distribution} $f(x)= (2\pi\lambda^2x^2)^{-1/2} \exp(-(\log(x)-\mu)/2\lambda^2)\Ii_{\IR_+}(x)$, $\lambda,x\in \IR_+$ and $\mu \in \IR$ delivers
    \begin{align*}
        |\sM_{1/2}[S](t)|=|(1/2+it)|^{-1} |\sM_{3/2}[f](t)| \leq C_{\mu,\lambda} |t|^{-1}\exp(-\lambda^2t^2/2), \text{ for } |t|\geq 1.
    \end{align*}
\end{enumerate}
\end{exs}
In all three cases, we can bound the bias term by $\|S-S_k\|\leq C \exp(-\delta k^r)$ for some $\delta, r\in \IR_+$ leading to a much sharper bound than $L k^{-2s},$ although it is easy to verify that all three examples lie in $\IW^s_{1/2}(L)$ for any $s\in \IR_+$ and $L\in \IR_+$ large enough. For example, in the case \textbf{(I)} the choice of $k_n=n^{1/(2s+2\gamma-1)}$ which is suggested in Proposition \ref{prop:upper:minimax} can be improved for any choice of $s\in \IR_+$. Setting $k_n=(\log(n)\delta^{-1})^{1/r}$ leads to 
\begin{align*}
    \IE(\|\widehat S_{k_n}-S\|^2) \leq C(f,g) \frac{\log(n)^{(2\gamma-1)/r}}{n} 
\end{align*}
which results in a sharper rate then $n^{-2s/(2s+2\gamma-1)}$.\\
Furthermore, despite the fact that the choice of $k_n$ in Proposition \ref{prop:upper:minimax} is not dependent on the explicit density $f\in \IW^s_{1/2}(L)$, it is still dependent on the regularity parameter $s\in \IR_+$ of the unknown density $f$ which is unknown, too. While it is tempting to set the regularity parameter $s\in \IR_+$ to a fixed value and interpret this as an additional model assumption, the discussion above motivates that this might deliver worse rates. In the next section, we therefore present a data-driven method in order to choose the parameter $k\in \IR_+$ based only on the sample $Y_1, \dots, Y_n$. 

\section{Data-driven method}
\label{sec_datadriven}

We now present a data-driven method based on a penalized contrast approach. In fact, we consider the case where \textbf{[G1]} holds with $\gamma>1/2$, since in the case $\gamma\leq 1/2$ we already presented a choice of the parameter $k\in \IR_+$, independent on the density $f$, which achieves an almost parametric rate. In the case $\gamma>1/2$, the second summand in Theorem \ref{thm:upp_bound} dominates the third term. Thus, the growth of the variance term is determined by the growth of $\Delta_g$.
Our aim is now to define an estimator $\widehat k_n$ which mimics the behavior of 
\begin{align*}
    k_n:\in\argmin\{\|S-S_k\|^2 + \IE(Y)C_g(2\pi n)^{-1} k^{2\gamma-1} : k \in \mathcal K_n\}
\end{align*} for a suitable large set of parameters $\mathcal K_n\subset \IR_+$. 
Considering the result of Proposition \ref{prop:upper:minimax}, and the fact that $\|S-S_k\|^2\leq k^{-1} \IE(X^{1/2})^2$, which we have seen in the paragraph about the parametric case, we can ensure that the set $\mathcal K_n:=\{k\in \{1,\dots, n\}: \Delta_g(k)\leq n^{-1}\}$ is suitably large enough. 
Starting with the bias term we see that $\|S-S_k\|^2 = \|S\|^2 - \|S_k \|^2$ behaves like  $-\|S_k\|^2$. Furthermore, for $k\in \mathcal K_n$ we define the penalty term $\mathrm{pen}(k)=\chi\sigma_Y \Delta_g(k)n^{-1},$ $\sigma_Y :=\IE(Y_1), $ which shall mimic the behavior of the variance term. Exchanging $-\|S_k\|^2$ and $\IE(Y_1)$ with their empirical counterparts $-\|\widehat S_k\|^2$ and $\widehat \sigma_Y:= n^{-1} \sum_{j=1}^n Y_j$ we define a fully data-driven model selection $\widehat k$  by   
\begin{align}\label{eq:data:driven}
\widehat k \in \argmin \{-\|\widehat S_k\|^2 + \widehat{\mathrm{pen}}(k) : k \in \mathcal K_n\} \quad \text{where} \quad\widehat{\mathrm{pen}}(k):= 2\chi \widehat \sigma_Y \Delta_g(k) n^{-1}
\end{align}
for $\chi>0.$
The following theorem shows that this procedure is adaptive up to a negligeable term.

\begin{thm}\label{dd:thm:ada}
	Let $g$ satisfy \textbf{[G1]} with $\gamma>1/2$ and $\|xg\|_{\infty}<\infty$. Assume further that $\IE(Y_1^{5/2})<\infty$. Then for $\chi >  96$,
	\begin{align*}
	\IE (	\| S- \widehat S_{\widehat k} \|^2) &\leq 6\inf_{k\in\mathcal K_n}\big(\|S-S_k\|^2 +\mathrm{pen}(k) \big) +  C(g,f)\left(n^{-1} +\Var(\widehat\sigma_X)+ \int_{-n}^n \Var(\sM_{1/2}[\widehat S_X](t))dt\right)
	\end{align*}
	where $C(g,f)>0$ is a constant depending on $\chi$, the error density $g$, $\IE(X_1^{5/2})$, $\sigma_X:=\IE(X_1)$ and $\widehat \sigma_X:= n^{-1}\sum_{j=1}^n X_j$. 
\end{thm}

The proof of Theorem \ref{dd:thm:ada} is postponed to Appendix. The assumption that $\|xg\|_{\infty}<\infty$ is rather weak because it is satisfied for the common densities considered. 

\begin{cor} \label{thm:upper_bound_adap}
	Let the assumptions of Theorem \ref{dd:thm:ada} hold. Then,
	\begin{enumerate}
	    \item[\textbf{(I)}] in the presence of i.i.d observations $X_1,...,X_n$ and $\IE[X_1^2] < \infty$,
	    \[
	        \IE (	\| S- \widehat S_{\widehat k} \|^2) \leq 6\inf_{k\in\mathcal K_n}\big(\|S-S_k\|^2 +\mathrm{pen}(k) \big) +  \frac{C(g,f)}{n}\left(1 +\IE[X_1^2]+\IE|X_1|\right);
	    \]
	    \item[\textbf{(B)}] for $\beta$-mixing observations $X_1,...,X_n$ under the additional assumptions that $\IE[X_1^2b(X_1)] < \infty$,
	    \[
	        \IE (	\| S- \widehat S_{\widehat k} \|^2) \leq 6\inf_{k\in\mathcal K_n}\big(\|S-S_k\|^2 +\mathrm{pen}(k) \big) +  \frac{C(g,f)}{n}\left(1 +\IE[X_1^2b(X_1)]+\IE[|X_1|b(X_1)]\right);
	   \]
	    \item[\textbf{(F)}] and for Bernoulli-shift processes \reff{def:bernoulli-shift} under the dependence measure \reff{def:functional_dependence_measure} provided that also $\sum_{j=1}^n\delta_1^X(j)^{1/2} < \infty$ and $\sum_{j=1}^n\delta_2^X(j) < \infty$,
	    \[
	        \IE (	\| S- \widehat S_{\widehat k} \|^2) \leq 6\inf_{k\in\mathcal K_n}\big(\|S-S_k\|^2 +\mathrm{pen}(k) \big) +  \frac{C(g,f)\log(n)}{n}\left(1 +\Big(\sum_{j=1}^n\delta_2^X(j)\Big)^2+\Big(\sum_{j=1}^n\delta_1^X(j)^{1/2}\Big)^2\right).
	   \]
	\end{enumerate}
\end{cor}

Under the assumptions of Theorem \ref{dd:thm:ada} and for the case $S\in \IW^{s}_{1/2}(L)$, we can ensure that $s>1/2$. Then we have $k_n:= \lfloor n^{1/(2s+2\gamma-1}\rfloor \in \mathcal K_n$ and thus
\begin{align*}
    \IE (	\| S- \widehat S_{\widehat k} \|^2) \leq C(f,g)n^{-2s/(2s+2\gamma-1)}
\end{align*}
for all three cases \textbf{(I)}, \textbf{(B)} and \textbf{(F)}.

\section{Numerical studies}
\label{sec_numerical}

In this section, we use Monte-Carlo simulation to visualise the properties of the estimator $\widehat S_{\widehat k}$. We will first consider the case of independent observations, \textbf{[I]}, and will then go on by study the behaviour of the estimator in presence of dependence. More precisely, we will simulate data following a functional dependency, that is \textbf{[F]}.

\subsection{Independent data}
Let us illustrate the performance of the fully-data driven estimator $\widehat S_{\widehat k}$ defined in \reff{eq:estim:def} and \reff{eq:data:driven}. To do so, we consider the following densities whose corresponding survival function will be estimated.
    \begin{enumerate}[(i)]
        \item \textit{Gamma distribution:} $f_1(x)=\frac{0.5^4}{\Gamma(4)}x^{3}\exp(-0.5x)\Ii_{\IR_+}(x)$,
        \item \textit{Weibull distribution:} $f_2(x)=2x\exp(-x^2)\Ii_{\IR_+}(x)$,
        \item \textit{Beta distribution:} $f_3(x)=\frac{1}{560}(0.5x)^3(1-0.5x)^4\Ii_{(0, 2)}(x)$ and 
        \item \textit{Log-Gamma distribution:} 
        $f_4(x)=\frac{x^{-4}}{6} \log(x)^3\Ii_{(1, \infty)}(x)$.
    \end{enumerate}
    For the distribution of the error density, we consider the following three cases
    \begin{enumerate}[(a)]
        \item \textit{Uniform distribution:} $g_1(x)=\Ii_{[0,1]}(x)$,
        \item \textit{Symmetric noise:} $g_2(x)=\Ii_{(1/2, 3/2)}(x)$ and
        \item \textit{Beta distribution:} $g_3(x)=2(1-x) \Ii_{(0,1)}(x)$.
    \end{enumerate}
    We see that $g_1$ and $g_2$ fulfill \textbf{[G1]} with the parameter $\gamma=1$ and $g_3$ fulfills it with $\gamma= 2$.
    Due to the fact that for the true survival function holds $S(x)\in [0,1], x\in \IR,$ we can improve the estimator $\widehat S_{\widehat k}$ by defining
    \begin{align*}
        \widetilde S_{\widehat k}(x):= \begin{cases} 0&, \widehat S_{\widehat k}(x)\leq 0;\\
        \widehat S_{\widehat k}(x) &, \widehat S_{\widehat k}(x) \in [0,1]; \\
        1&, \widehat S_{\widehat k}(x) \geq 1.
        \end{cases}
    \end{align*}
    The resulting estimator $\widetilde S_{\widehat k}$ has a smaller risk then $\widehat S_{\widehat k}$, since $\| \widetilde S_{\widehat k}-S\|^2 \leq \|\widehat S_{\widehat k}-S\|$. On the other hand, the estimator has the desired property that it is $[0,1]$-valued. Nevertheless, it is difficult ensure that $\widetilde S_{\widehat k}(0)=1$ and that $\widetilde S_{\widehat k}$ is monotone decreasing.
    Although there are many procedures to guarantee the monotonicity of an estimator $\widehat S$ and the property $\widehat S(0)=1$, the presented theoretical results of this work are not applicable to the modified estimators. \\
    We will now use a Monte-Carlo simulation to visualize the properties of the estimator $\widetilde S_{\widetilde k}$ and discuss whether the numerical simulated behavior of the estimator coincides with the theoretical predictions. \\
    After that we will construct a survival estimator $\widehat{S}$ based on $\widehat S_{\widehat k}$ which is in fact a survival function, keeping in mind that the theoretical results of this work do not apply for this estimator.

\begin{figure}[ht]
\centering
\begin{tabular}{@{}cc||c|c|c@{}}
			&& $n=500$ & $n=1000$ & $n=2000$\\\midrule\midrule
			&(i)  & $1.31$&$0.75$ &$0.23$\\
			& (ii)& $0.19$ & $0.09$ & $0.04$\\
			&(iii)& $0.21$& $0.12$& $0.06$\\
			&(iv)& $1.10$& $0.34$ &$0.20$\\\bottomrule
	\end{tabular}\vspace*{0.2cm}
	\caption{The entries showcase the MISE (scaled by a factor of 100) obtained by Monte-Carlo simulations each with 200 iterations. We take a look at different densities and varying sample size. The error density is chosen as (a) in each case.}
\label{table:1}
\end{figure}

    In Figure \ref{table:1} we can see that for an increasing sample size, the variance of the estimator seems to decrease. Also,, increasing the sample size allows for the estimator to choose bigger values for $\widehat k$, which on the other hand decreases the bias induced by the approximation step. Next let us consider different error densities.
    
       \begin{minipage}{\textwidth}
\centering{\begin{minipage}[t]{0.32\textwidth}
		\includegraphics[width=\textwidth,height=60mm]{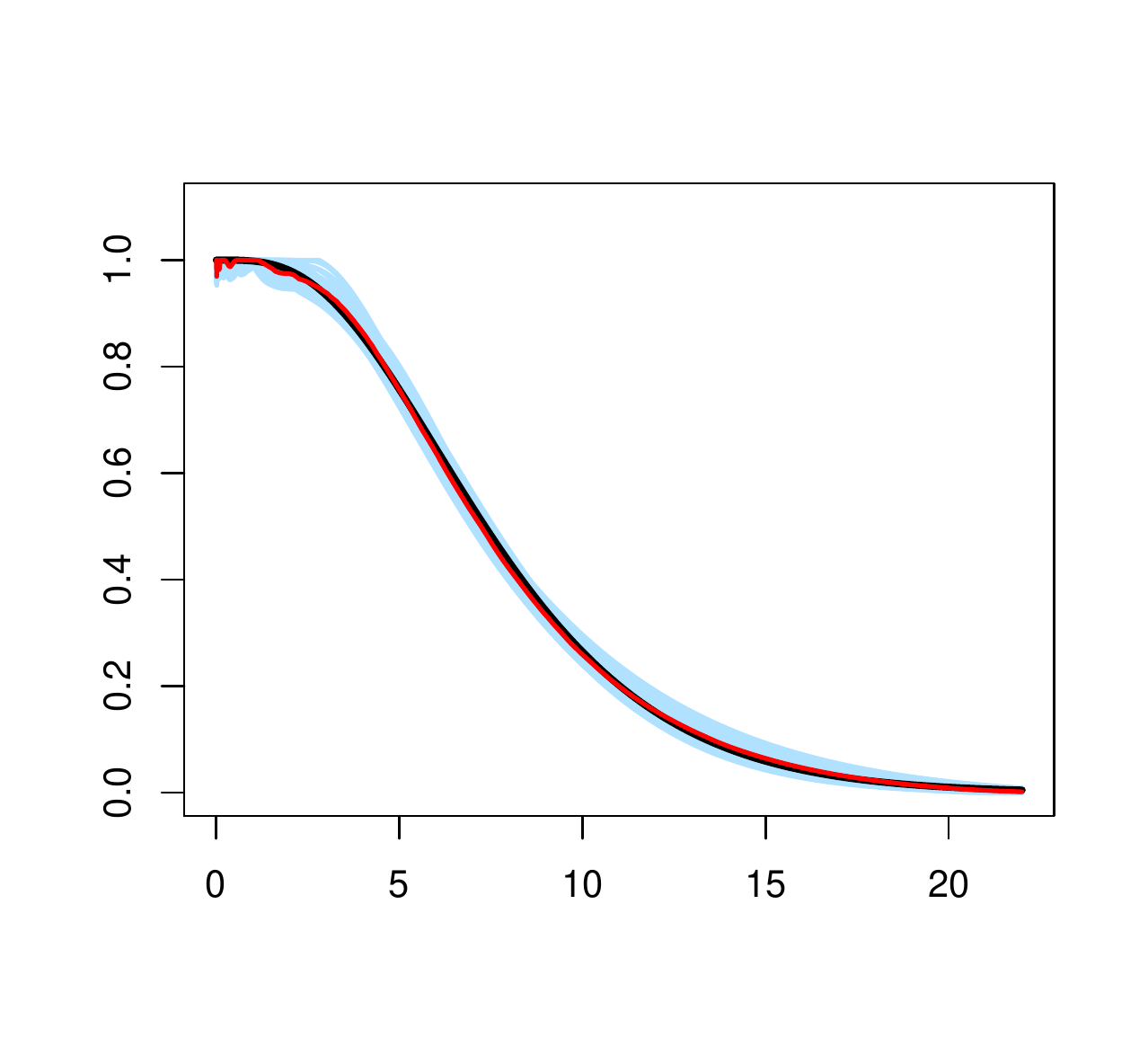}
	\end{minipage}
	\begin{minipage}[t]{0.32\textwidth}
		\includegraphics[width=\textwidth,height=60mm]{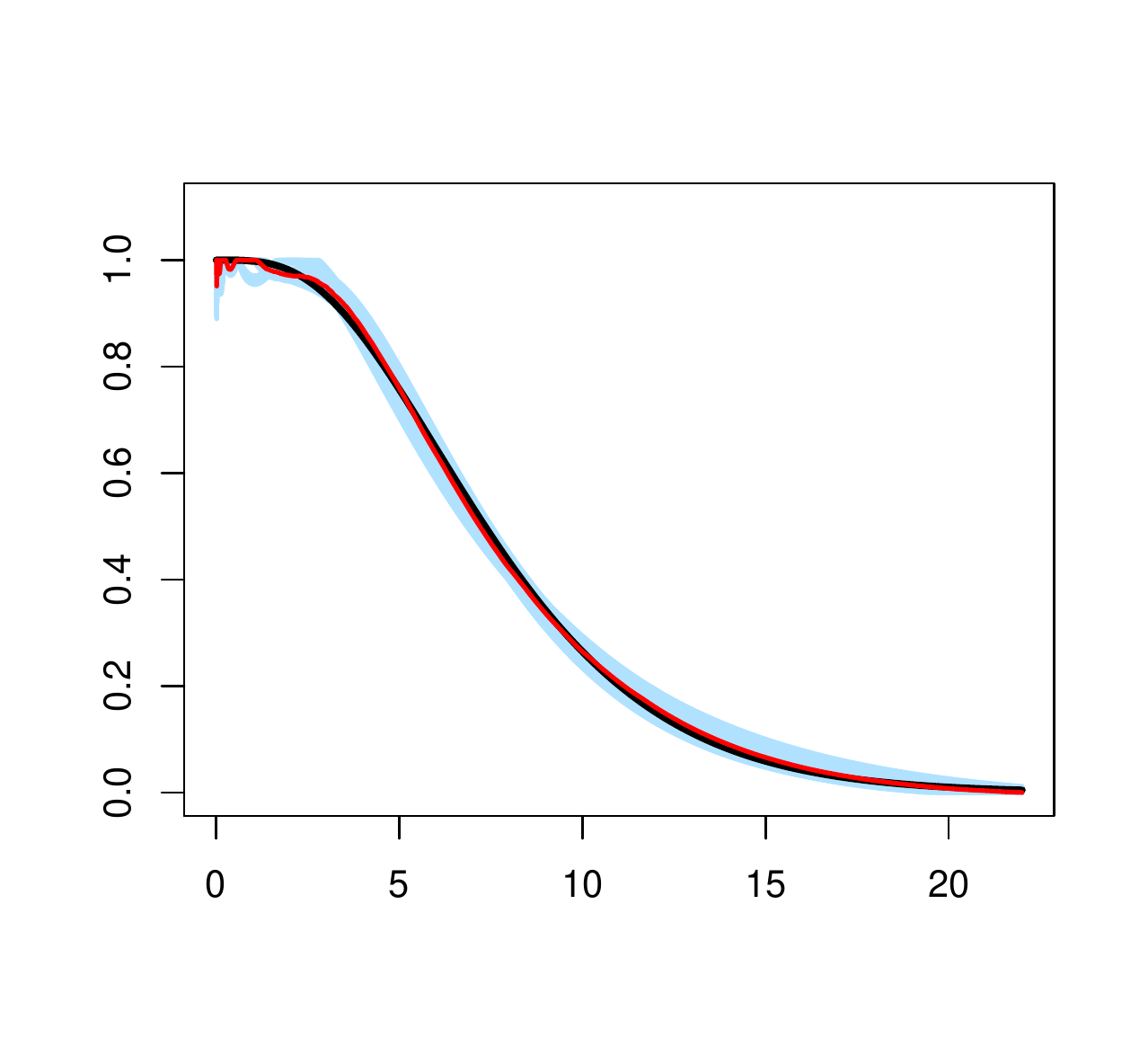}
	\end{minipage}
	\begin{minipage}[t]{0.32\textwidth}
		\includegraphics[width=\textwidth,height=60mm]{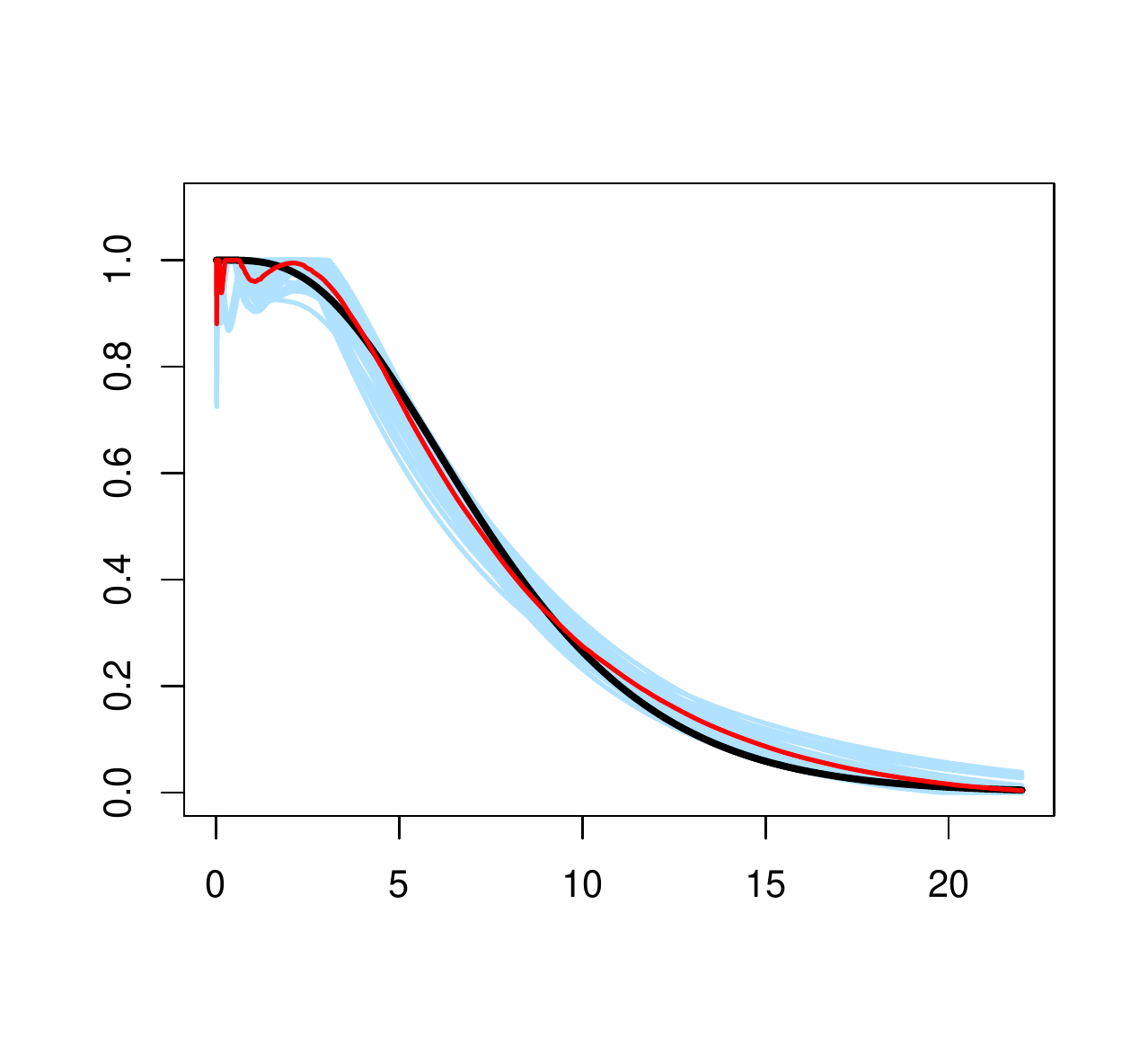}
	\end{minipage}}
	
\centering{\begin{minipage}[t]{0.32\textwidth}
		\includegraphics[width=\textwidth,height=60mm]{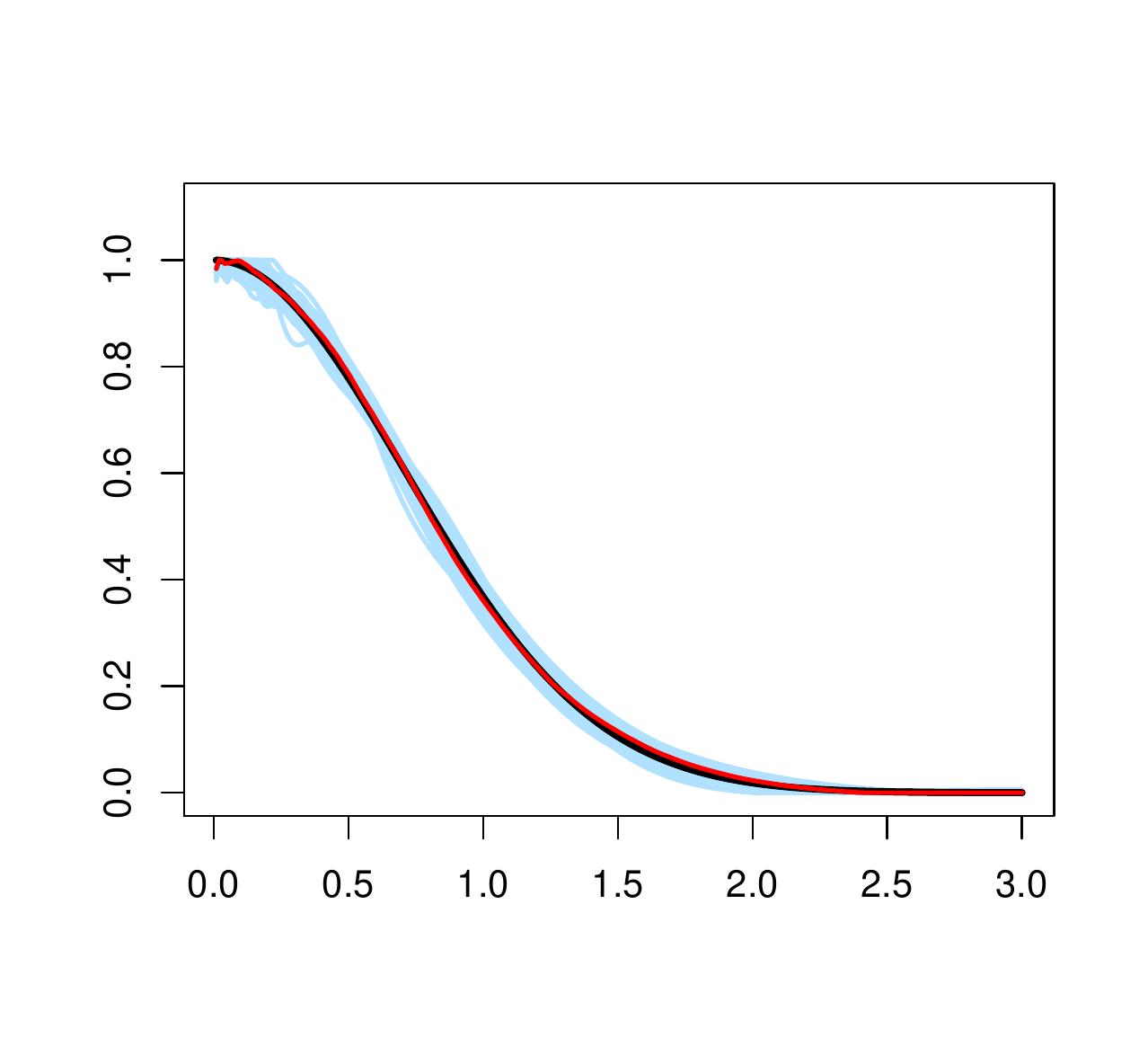}
	\end{minipage}
	\begin{minipage}[t]{0.32\textwidth}
		\includegraphics[width=\textwidth,height=60mm]{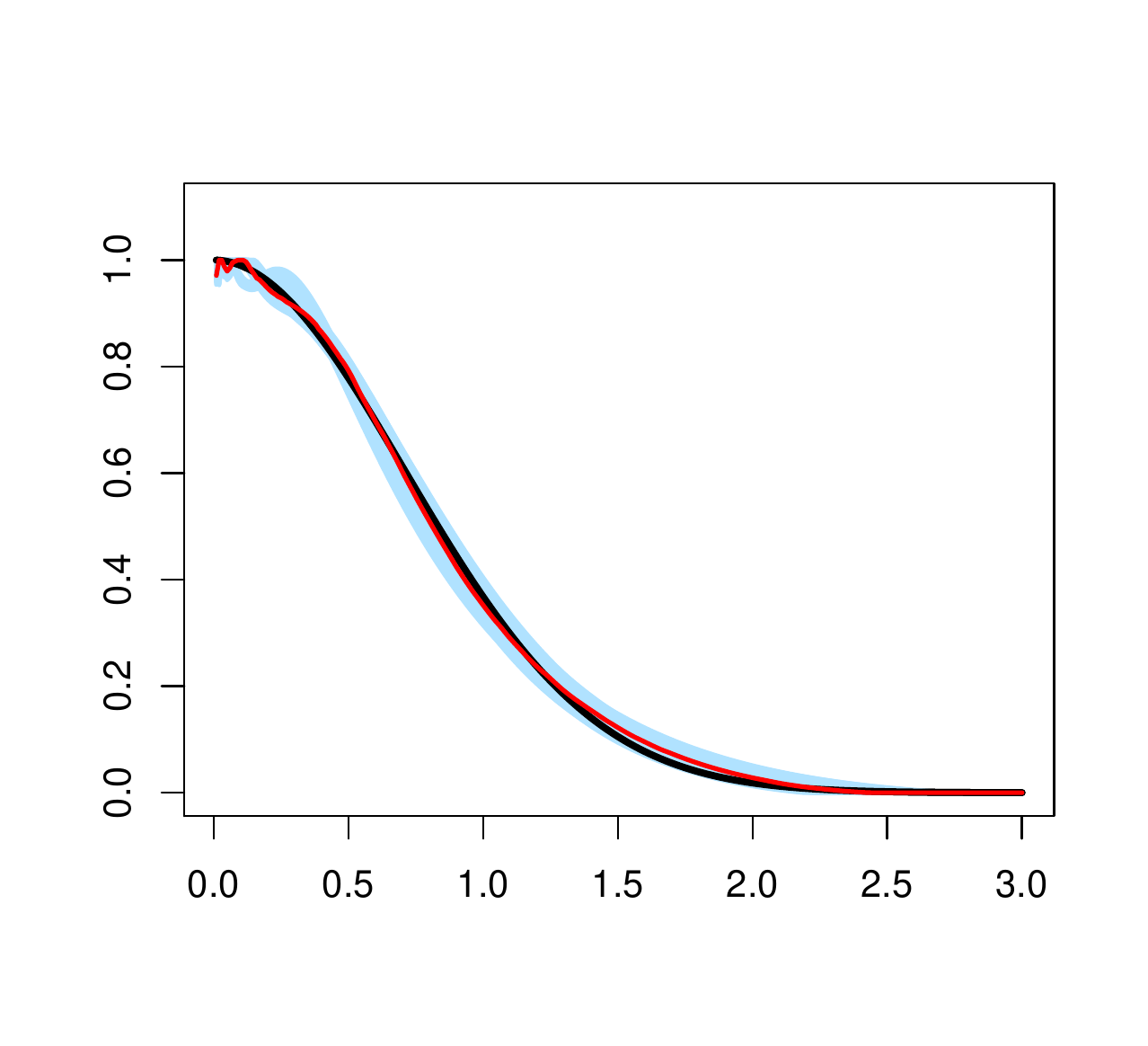}
	\end{minipage}
\begin{minipage}[t]{0.32\textwidth}
	\includegraphics[width=\textwidth,height=60mm]{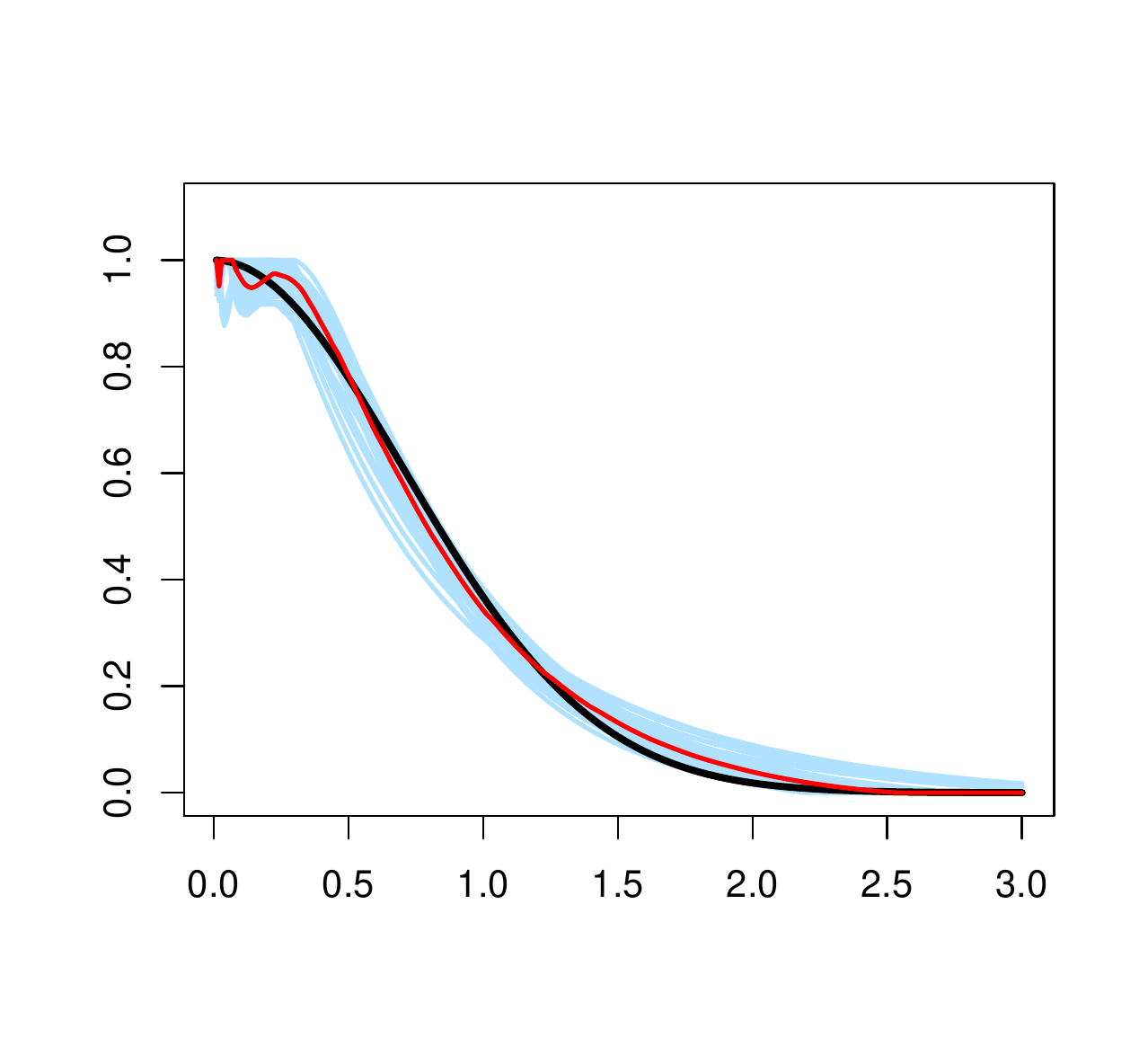}
\end{minipage}}
      \captionof{figure}{\label{figure:3}Considering 
        the estimators $\widetilde S_{\widehat k}$, we depict 
        50  Monte-Carlo simulations with varying error density (a) (left), (b) (middle) and (c) (right) for (i) (top), (ii) (bottom) with $n=1000$. The true survival function $ S$ is given by the black curve while the red curve is the point-wise empirical median of the 50 estimates.}
\end{minipage}\\[2ex]
In Figure \ref{figure:3} we can see that reconstruction of the survival function with error density (a) and (b) seem to be of same complexity while the reconstruction with error density (c) seems to be more difficult. This behavior is predicted by the theoretical results because for \textbf{[G1]}, (a) and (b) share the same parameter $\gamma=1$ while (c) has the parameter $\gamma=2$.

\paragraph{Heuristic estimator} We will now modify the estimator $\widehat S_{\widehat k}$ such that the resulting estimator $\widehat S$ is a survival function. To do so, we see that for any $k\in \IR$ and $x \in \IR_+$,

\begin{align*}
    \widehat S_k(x)=\frac{1}{2\pi} \int_{-k}^k x^{-1/2-it} \frac{\widehat{\sM}(t)}{\sM_{3/2}[g](t)}(1/2+it)^{-1} dt = (\widehat p_k * g_u)(x),
\end{align*}
where $g_u(x)=\Ii_{(0,1)}(x)$ the density of the uniform distribution on the intervall $(0,1)$ and $\widehat p_k(x)= (2\pi)^{-1} \int_{-k}^k x^{-1/2-it} \frac{\widehat \sM(t)}{\sM_{3/2}[g](t)} dt$. Exploiting the definition of the multiplicative convolution we see that for any $x \in \IR_+$,
\begin{align*}
    \widehat S_k(x)= \int_{x}^{\infty} \widehat p_k(y) y^{-1}dy.
\end{align*}
This motivates the following construction of the survival function estimator. First, exchanging $\widehat p_k $ with $(\widehat p_k(x))_+$ ensures the monotonicity and the positivity of our estimator. The final estimator is then defined as
\begin{align*}
    \widehat S(x):= \widetilde S(x)/\widetilde S(0+), \quad \text{where }\widetilde S(x):= \int_x^{\infty} (\widehat p_{\widehat k}(y))_+ y^{-1}dy, \text{ for any } y\in \IR_+.
\end{align*}
where $0+$ denotes a positive real number very close to 0. Since our estimator is not defined in $0$ we cannot normalize it with $\widetilde S(0)$.
    
Let us now illustrate the behavior of the heuristic estimator $\widehat S$ for an increasing number of observations compared to the estimator $\widetilde S_{\widehat k}$.\\
  \begin{minipage}{\textwidth}
\centering{\begin{minipage}[t]{0.32\textwidth}
		\includegraphics[width=\textwidth,height=60mm]{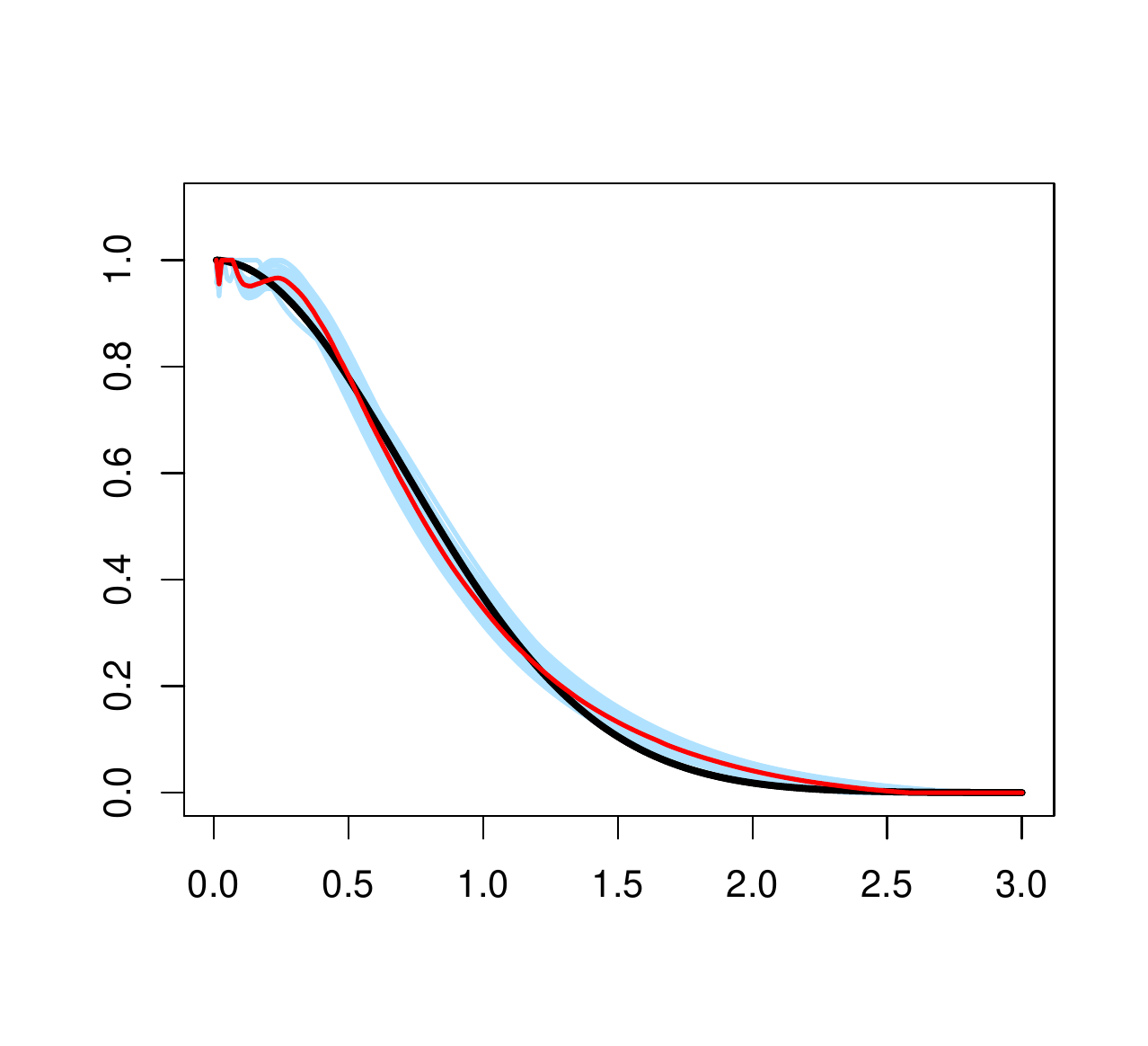}
	\end{minipage}
	\begin{minipage}[t]{0.32\textwidth}
		\includegraphics[width=\textwidth,height=60mm]{02_02_1000.pdf}
	\end{minipage}
	\begin{minipage}[t]{0.32\textwidth}
		\includegraphics[width=\textwidth,height=60mm]{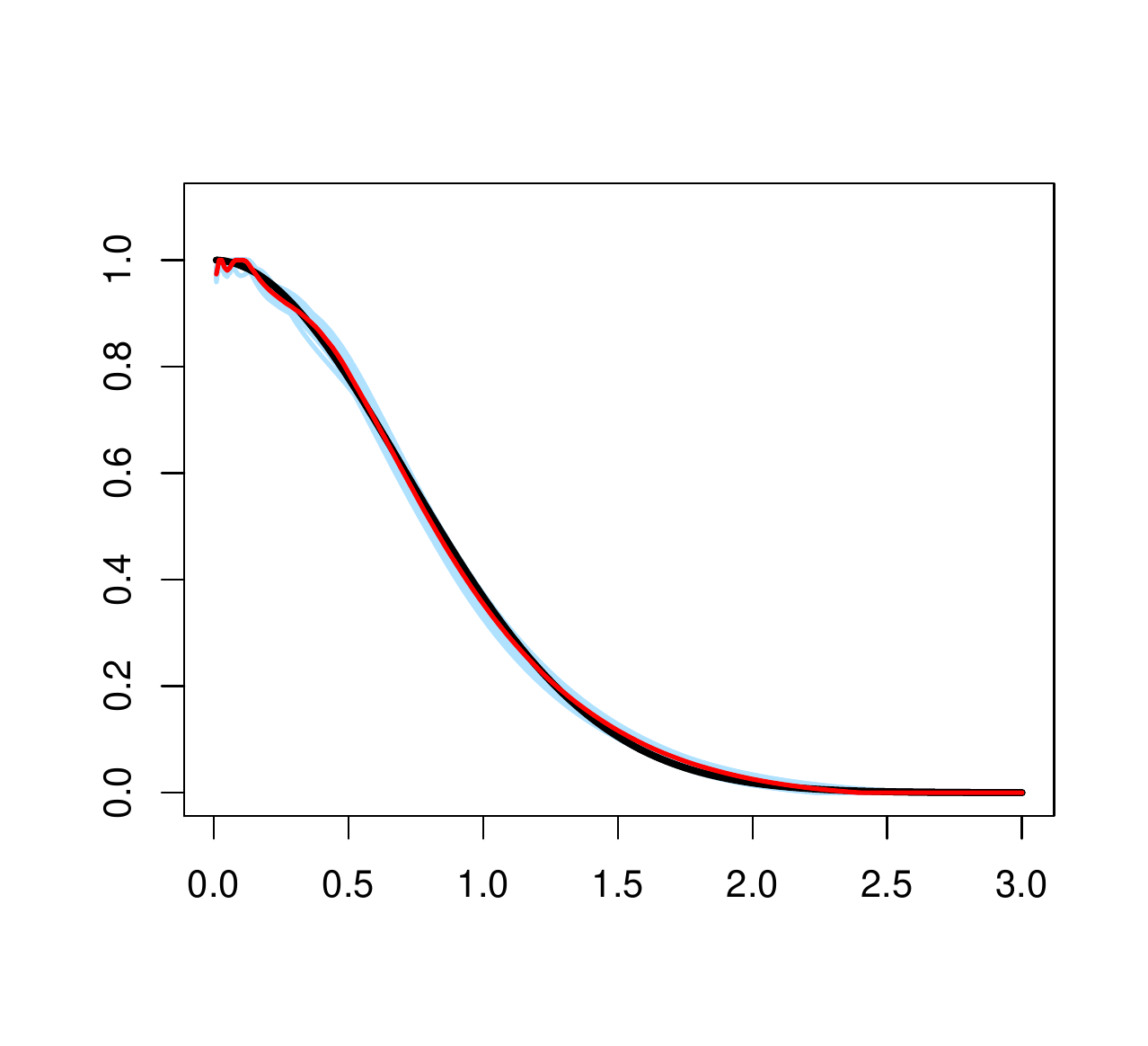}
	\end{minipage}}
	
\centering{\begin{minipage}[t]{0.32\textwidth}
		\includegraphics[width=\textwidth,height=60mm]{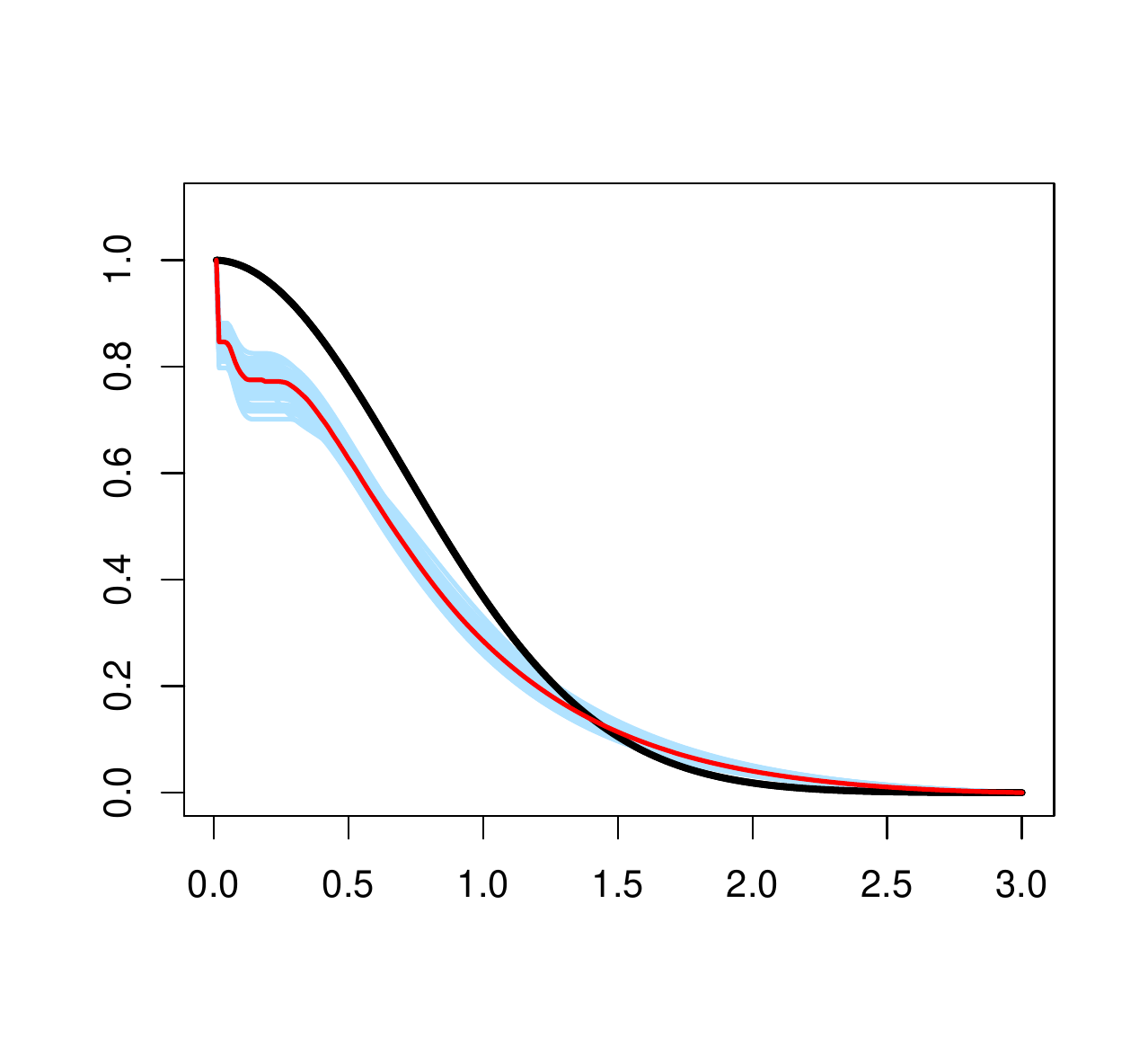}
	\end{minipage}
	\begin{minipage}[t]{0.32\textwidth}
		\includegraphics[width=\textwidth,height=60mm]{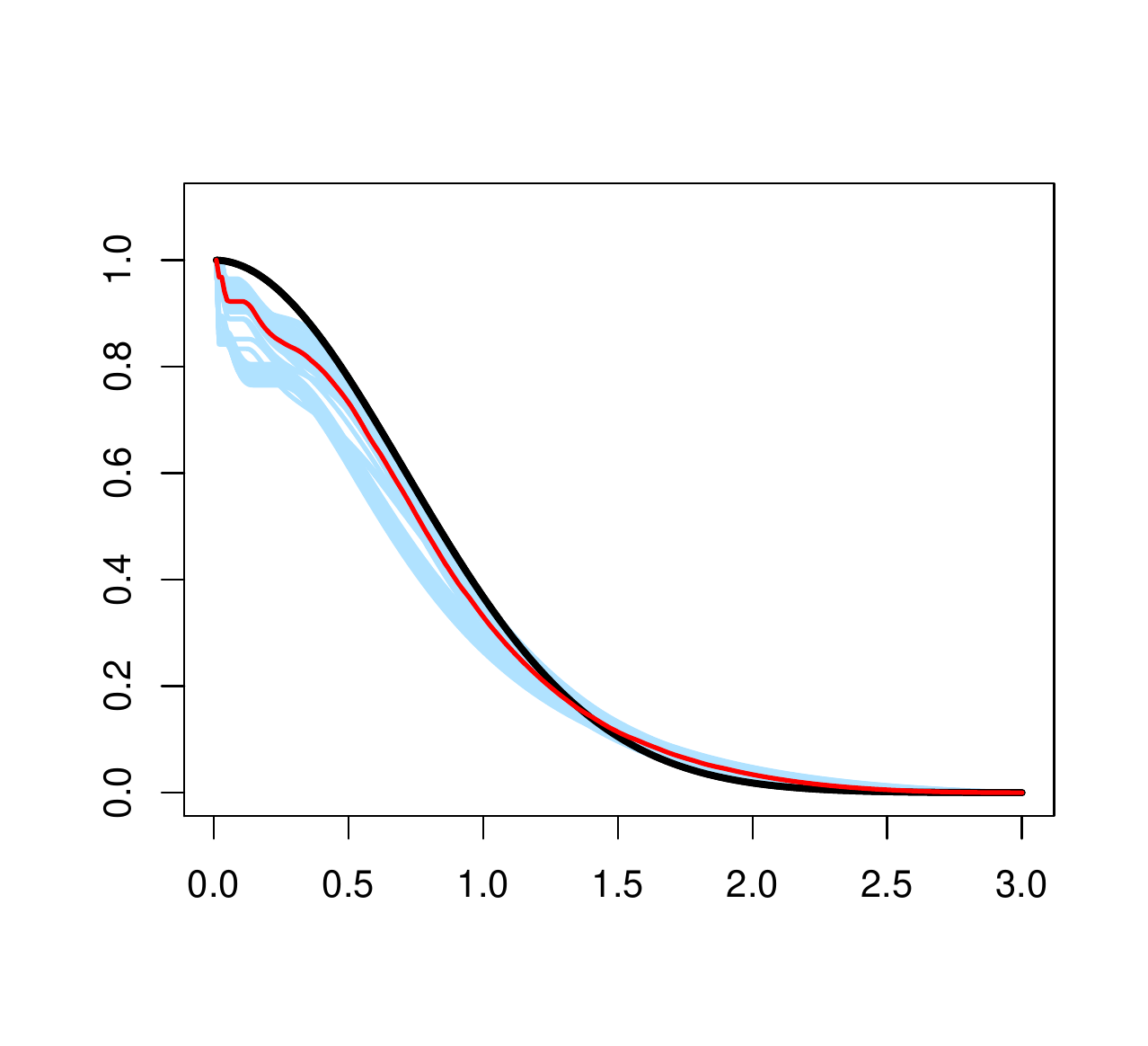}
	\end{minipage}
\begin{minipage}[t]{0.32\textwidth}
	\includegraphics[width=\textwidth,height=60mm]{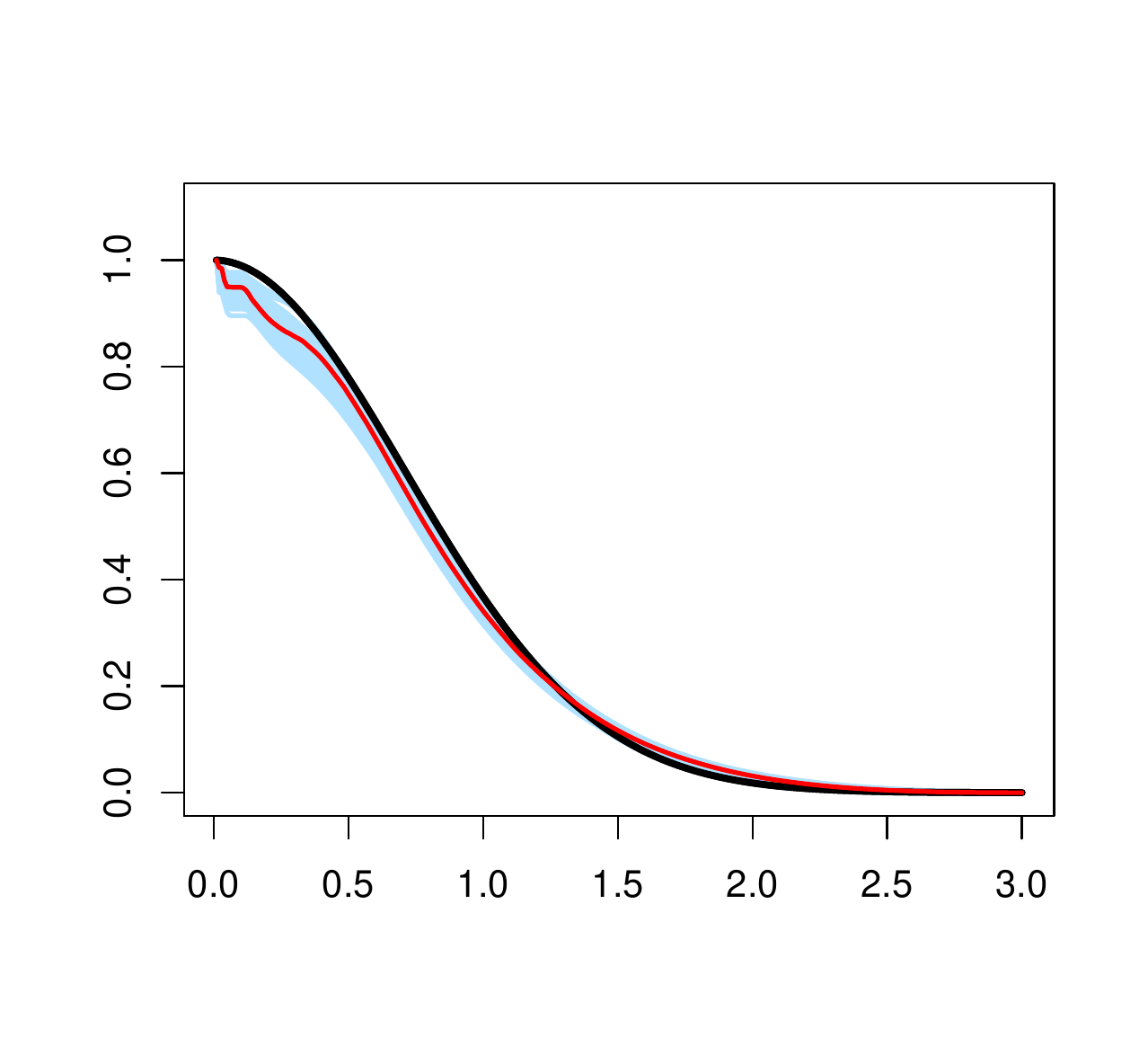}
\end{minipage}}
      \captionof{figure}{\label{figure:4}Considering 
        the estimators $\widetilde S_{\widehat k}$ (top) and $\widehat S$ (bottom), we depict 
        50 Monte-Carlo simulations with varying sample size $n=500$ (left), $n=1000$ (middle) and $n=2000$ (right) in the case (ii) with error density (b). The true survival function $S$ is given by the black curve while the red curve is the point-wise empirical median of the 50 estimates.}
\end{minipage}
Although the estimator $\widehat S$ is a survival function, the modification of the estimator seems to introduce an additional bias to the estimation. Based on the numerical study, this additional bias seems to decrease for $n$ large enough. Nevertheless, the authors want to stress out that this modification is purely heuristic. 
    
\paragraph{Dependent data}
In this section we are going to take a look at an AR(1)-process of the following form and analyze its structure using the functional dependence measure. For $|\rho| < 1$,
    \begin{align*}
        X_n :=\rho X_{n-1} + \varepsilon_n, \ \text{ where  } \ \varepsilon_n|B_n\sim \Gamma_{(B_n, \lambda)}\ \text{ with }\  B_n \sim \mathrm{Bin}_{(m, 1-\rho)}.
    \end{align*}
It is apparent that the innovations constructed are i.i.d. and for this specific choice of $\varepsilon_n$, the marginals of the AR(1)-process follow a Gamma distribution $\Gamma_{(m,\lambda)}$ as shown in \cite{gaver_lewis}. On the other hand, $(X_n)$ then emits a representation as a Bernoulli shift $X_n = \sum_{j=0}^\infty \rho^j\varepsilon_{n-j}$ (cf. \cite{brockwell} or\cite{chow_teicher}). In this case,
\[
    \delta_p^X(k) = |\rho^k|\sup_{j=1,...,n}\|\varepsilon_j-\varepsilon_j^{*(j-k)}\|_p \le 2|\rho^k|\|\varepsilon_1\|_p.
\] Therefore, $\sum_{k=0}^\infty \delta_2^X(k) = \frac{1}{1-|\rho|} < \infty$ and $\sum_{k=0}^\infty \delta_1^X(k)^{1/2} = \frac{1}{1-|\rho|^{1/2}}< \infty$ as geometric series and satisfies the assumptions of Corollaries \ref{cor:consis_dep} and \ref{thm:upper_bound_adap}. Specifically, the estimator's variance depends on the choice of $\rho$. The closer $\rho$ is to one, the more dependent the data becomes where as the variance increases. However, if $\rho$ approaches zero, the data can be seen more independent until it becomes i.i.d. This behavior can be seen in Figure \ref{table:2}, as well.
   
\begin{figure}[ht]
\centering
\begin{tabular}{@{}cc||c|c|c@{}}
			&& $n=500$ & $n=1000$ & $n=2000$\\\midrule\midrule
			$m=1$&$\rho=0.1$ & $0.15$&$0.11$ &$0.05$\\
			&$\rho=0.5$& $0.29$ & $0.16$ & $0.07$\\
			&$\rho=0.9$& $1.19$& $0.63$& $0.35$\\
			$m=4$&$\rho=0.1$& $0.69$& $0.39$ &$0.13$\\
			&$\rho=0.5$& $0.93$& $0.52$ &$0.18$\\
			&$\rho=0.9$& $2.87$& $1.57$ &$0.70$\\
			\bottomrule
	\end{tabular}\vspace*{0.2cm}
	\caption{The entries showcase the MISE (scaled by a factor of 100) obtained by Monte-Carlo simulations each with 200 iterations. We take a look at different densities, three distinct sample sizes and varying $\rho$. The error density is chosen as (a) in each case.}
\label{table:2}
\end{figure}

\newpage
\section{Appendix}
\label{sec_append}

\subsection{Useful inequalities}

The following inequality is due to
\cite{Talagrand1996}, the formulation of the first part  can be found for
example in \cite{KleinRio2005}.

\begin{lem}(Talagrand's inequality)\label{tal:re} Let
	$Z_1,\dotsc,Z_n$ be independent $\mathcal Z$-valued random variables and let $\bar{\nu}_{h}=n^{-1}\sum_{i=1}^n\left[\nu_{h}(Z_i)-\IE\left(\nu_{h}(Z_i)\right) \right]$ for $\nu_{h}$ belonging to a countable class $\{\nu_{h},h\in\sH\}$ of measurable functions. Then,
	\begin{align}
	&\IE\left(\sup_{h\in\sH}|\bar{\nu}_{h}|^2-6\Psi^2\right)_+\leq C \left[\frac{\tau}{n}\exp\left(\frac{-n\Psi^2}{6\tau}\right)+\frac{\psi^2}{n^2}\exp\left(\frac{-K n \Psi}{\psi}\right) \right]\label{tal:re1} 
	\end{align}
	with numerical constants $K=({\sqrt{2}-1})/({21\sqrt{2}})$ and $C>0$ and where
	\begin{equation*}
	\sup_{h\in\sH}\sup_{z\in\mathcal Z}|\nu_{h}(z)|\leq \psi,\qquad \IE(\sup_{h\in\sH}|\overline{\nu}_{h}|)\leq \Psi,\qquad \sup_{h\in \sH}\frac{1}{n}\sum_{j=1}^n \Var(\nu_{h}(Z_i))\leq \tau.
	\end{equation*}
\end{lem}
\begin{rem}
	Keeping the bound \eqref{tal:re1}  in mind, let us
	specify particular choices $K$; in fact $K\geq \tfrac{1}{100}$.
	The next bound there becomes an immediate consequence, 
	\begin{align}
	&\IE\left(\sup_{h\in\sH}|\bar{\nu}_{h}|^2-6\Psi^2\right)_+\leq C \left(\frac{\tau}{n}\exp\left(\frac{-n\Psi^2}{6\tau}\right)+\frac{\psi^2}{n^2}\exp\left(\frac{-n \Psi}{100\psi}\right) \right).\label{tal:re3} 
	\end{align}
\end{rem}

\subsection{Proofs of Section \ref{sec_minimax}}\label{a:mt}
\begin{proof}[Proof of Lemma \ref{prop:surv}]
Let us assume that $\IE(X^{1/2})<\infty$. Then, $f\in \IL^1(\IR,x^{1/2})$ and thus $S\in \IL^1(\IR, x^{-1/2})$ because
\begin{align*}
    \int_0^{\infty}x^{-1/2} S(x) dx = \int_0^{\infty} \int_0^{\infty} x^{-1/2}\Ii_{(0,x)}(y) f(y) dy dx = \int_0^{\infty} 2y^{1/2} f(y) dy<\infty.
\end{align*}
By the generalized Minkowski inequality, cf. \cite{Tsybakov2008}, we have
\begin{align*}
\int_0^{\infty} S^2(x)dx &= \int_0^{\infty} \left( \int_0^{\infty} \Ii_{(0,x)}(y)
f(x)dx \right)^2 dy\\
&\leq \left( \int_0^{\infty} \left(\int_0^{\infty} \Ii_{(0,x)}(y) f^2(x) dy \right)^{1/2} dx \right)^2 
= \left(\int_0^{\infty} x^{1/2} f(x) dx\right)^2,
\end{align*}
which implies that $S\in \IL^2(\IR_+, x^0)$.
For the second part we can easily see that $\IE(\|\widehat S_X\|_{\IL^1(\IR_+,x^{-1/2})})= \int_0^{\infty}x^{-1/2} S_X(x)dx <\infty$. Since $\IE(X_1)<\infty$, we get \begin{align*}
    \IE(\|\widehat S_X\|^2)\leq \int_0^{\infty} \IE(\Ii_{(x,\infty)}(X_1))  dx = \int_0^{\infty} S(x)dx = \IE(X_1)< \infty.
\end{align*} This implies that $\widehat S_X \in\IL^1(\IR_+,x^{-1/2}) \cap \IL^2(\IR_+,x^0)$ almost surely.
\end{proof}
\begin{proof}[Proof of Theorem \ref{thm:upp_bound}] Let $k\in \IR_+$. Since $\sM_{1/2}[S-S_k](t)=0$ for $|t|\leq k$ we get by application of the  Plancherel equation, $\langle S-S_k, S_k-\widehat S_k\rangle = \frac{1}{2\pi}\int_{-k}^k \sM_{1/2}[S-S_k](t) \sM_{1/2}[S_k-\widehat S_k](-t)dt = 0$ and thus $\|S-\widehat S_k\|^2= \|S-S_k\|^2+\|S_k-\widehat S_k\|^2.$ Again by application of the Plancherel equality, cf. equation \reff{eq:Mel:plan}, and the Fubini-Tonelli theorem,
\begin{align*}
    \IE(\|\widehat S_k-S_k\|^2) = \frac{1}{2\pi} \int_{-k}^k \frac{\Var(\widehat \sM(t))}{|(1/2+it)\sM_{3/2}[g](t)|^{2}}dt.
\end{align*}
Now we will use the fact that $\Var(\widehat \sM(t))= \Var(\widehat \sM(t)-\IE_{|X}(\widehat \sM(t)))+\Var(\IE_{|X}(\widehat \sM(t))$. For the first summand we see that by $\IE_{|X}(Y_j^{1/2+it})=\IE(U_j^{1/2+it})X_j^{1/2+it}$ and
\begin{align*}
\Var\left(\widehat \sM(t)-\IE_{|X}(\widehat \sM(t))\right)&= n^{-2} \sum_{j,j'=1}^n \IE(X_j^{1/2+it} X_{j'}^{1/2-it} (U_j^{1/2+it}-\IE(U_j^{1/2+it})(U_{j'}^{1/2-it}-\IE(U_{j'}^{1/2-it}))\\
&=n^{-2} \sum_{j,j'=1}^n \IE(X_j^{1/2+it} X_{j'}^{1/2-it})\IE( (U_j^{1/2+it}-\IE(U_j^{1/2+it})(U_{j'}^{1/2-it}-\IE(U_{j'}^{1/2-it}))\\
&=n^{-1}\IE(X_1)\Var(U_1^{1/2+it})\leq n^{-1} \IE(Y_1)
\end{align*}
On the other hand, we have $\IE_{|X}(\widehat{\sM}(t))=n^{-1} \sum_{j=1}^n X_j^{1/2+it} \IE(U_j^{1/2+it})$ and thus
\begin{align*}
    \Var(\IE_{|X}(\widehat{\sM}(t))
    &=|\IE(U_1^{1/2+it})|^2\IE(|n^{-1} \sum_{j=1}^n X_j^{1/2+it}- \sM_{3/2}[f](t)|^2).
\end{align*}
Due to these considerations and the fact that $\IE(U_1^{1/2+it})=\sM_{3/2}[g](t)$ for all $t\in \IR$, we have
\begin{align*}
    \IE(\|\widehat S_k-S_k\|^2) \leq \IE(Y_1) \frac{\Delta_g(k)}{n} + \frac{1}{2\pi} \int_{-k}^k \frac{\IE(|\widehat \sM_X(t)-\sM_{3/2}[f](t)|^2)}{|1/2+it|^2}dt
\end{align*}
As for the last summand, we see for any $k\in \IR_+$,
\begin{align*}
    \frac{1}{2\pi} \int_{-k}^k \frac{\IE(|\widehat \sM_X(t)-\sM_{3/2}[f](t)|^2)}{|1/2+it|^2}dt = \frac{1}{2\pi} \int_{-k}^{k}\Var(\sM_{1/2}[\widehat S_X](t)) dt 
\end{align*}
using equation \reff{eq:mel:surv}.
\end{proof}
For the proof of Corollary \ref{cor:consis_dep} we will need the following results. The key statement regarding $\beta$-mixing processes is delivered by the proposed variance bound derived by \cite{asin} after Lemma 4.1 of the same work. Their approach is based on the original idea of \cite[Theorem 2.1]{viennet}.

\begin{lem} \label{lem:beta_mix_var_bound}
    Let $(Z_j)_{j \in \IZ}$ be a strictly stationary process of real-valued random variables with common marginal distribution $\IP$. There exists a sequence $(b_k)_{k \in \IN}$ of measurable functions $b_k:\IR \to [0,1]$ with $\IE_{\IP}[b_k(Z_0)] = \beta(Z_0,Z_k)$ such that for any measurable function $h$ with $\IE[|h(Z_0)|^2] < \infty$ and $b = \sum_{k=1}^\infty (k+1)^{p-2}b_k : \IR \to [0,\infty]$, $p \ge 2$,
    \begin{equation*}
        \Var\big(\sum_{j=1}^n h(Z_j)\big) \le 4n \IE[|h(Z_0)|^2b(Z_0)]
    \end{equation*}
    where we set $b_0 \equiv 1$.
\end{lem}

As far as the functional dependence measure is concerned, we do not have access to the process $(X_j)$ itself. In its given form $(X_j^{1/2+it})$ we make use of the two following statements.

\begin{lem} \label{lem:holder}
The function $g: \IR_+ \to \IR, x \to x^{1/2 + it}$, $t \in \IR$, is Hölder continuous with exponent $1/2$, i.e. $|g(x)-g(y)| \le L(t)\cdot |x-y|^{1/2}$ where $L(t) = 1 + 4|t|^{1/2}$.
\end{lem}

\begin{lem} \label{lem:transfer_dependency}
    Let $g: \IR_+ \to \IR, x \to x^{1/2 + it}$, $t \in \IR$. For \textbf{(F)} holds then, $\Var(\sum_{j=1}^n g(X_j))^{1/2} \le L(t) \cdot n^{1/2}\sum_{k=1}^\infty \delta_1^X(k)^{1/2}$ where $L(t) = 1 + 4|t|^{1/2}$.
\end{lem}

\begin{proof}[Proof of Corollary \reff{cor:consis_dep}]
We consider the different dependence structures separately.
    \begin{enumerate}
        \item[\textbf{(I)}] The result for independent observations can directly be obtained since
        \begin{eqnarray*}
            \Var(\sM_{1/2}[\widehat S_X](t)) = |1/2 + it|^{-2}n^{-1}\Var(X_1^{1/2+it}) 
            \le |1/2 + it|^{-2}n^{-1}\IE(X_1).
        \end{eqnarray*} On the other hand, as $k \to \infty$, $(2\pi)^{-1}\int_{-k}^k|1/2+it|^{-2}dt\rightarrow 1$.
        \item[\textbf{(B)}] In the case of $\beta$-mixing we employ Lemma \reff{lem:beta_mix_var_bound}. We then have
        \begin{eqnarray*}
            \Var(\sM_{1/2}[\widehat S_X](t)) \le 4|1/2 + it|^{-2}n^{-2} \cdot n \IE[|X_1^{1/2+it}|^2b(X_1)] \le 4|1/2 + it|^{-2}n^{-1} \IE[|X_1|b(X_1)].
        \end{eqnarray*} As before, as $k \to \infty$, $(2\pi)^{-1}\int_{-k}^k|1/2+it|^{-2}dt\rightarrow 1$.
        \item[\textbf{(F)}] We now study the dependence measure. According to Lemma \reff{lem:transfer_dependency} we have
        \begin{eqnarray*}
            \Var(\sM_{1/2}[\widehat S_X](t)) dt \le L^2(t)|1/2 + it|^{-2} \cdot n^{-1} \big(\sum_{k=1}^\infty\delta_1^X(k)^{1/2}\big)^2 
        \end{eqnarray*} where $L(t) = 1 + 4 |t|^{1/2}$. Therefore,
        $
            \int_{-k}^k L^2(t)(1/2 + it)^{-2} dt \le c\log(k)
        $ for a numerical constant $c > 0$.
    \end{enumerate}
    
\end{proof}

\begin{proof}[Proof of Lemma \reff{lem:holder}]
Without loss of generality let $x > y > 0$. By the elementary inequality $|x^{1/2}-y^{1/2}| \le |x-y|^{1/2}$ we have
    \begin{eqnarray*}
        |g(x) - g(y)| = |x^{1/2 + it} - y^{1/2 + it}| &\le& |x^{it}|\cdot|x-y|^{1/2} + |y|^{1/2}\cdot|x^{it}-y^{it}|
        \\ &\le& |x-y|^{1/2} + |y|^{1/2}\cdot|x^{it}-y^{it}|
    \end{eqnarray*}
    We then bound the second term. First we see that
    \begin{eqnarray*}
        |y|^{1/2}\cdot|x^{it}-y^{it}| \le |y|^{1/2} \cdot (|\cos(t\log(x))-\cos(t\log(y))| + |(\sin(t\log(y))-\sin(t\log(y)))|)
    \end{eqnarray*}
    Moreover,
    \begin{eqnarray*}
        |\cos(t\log(x))-\cos(t\log(y))| &\le& |t|\cdot \log(x/y) =  |t|\cdot\log \Big(1 + \frac{x-y}{y}\Big) \\
        &\le& |t| \cdot \frac{x-y}{y}. 
    \end{eqnarray*} where we used $\log(1+z) \le z$ for $z > 0$. At the same time $|\cos(t\log(x))-\cos(t\log(y))| \le 2$.
    Applying both bounds together and exploiting $\min\{1,z\} \le z^s$ for $z \ge 0$, $s \in (0,1)$,
    \[
        |y|^{1/2}\cdot|\cos(t\log(x))-\cos(t\log(y))| \le 2|y|^{1/2}\min\Big\{|t|\frac{x-y}{y},1\Big\} \le 2|t|^{1/2}|y|^{1/2}\Big|\frac{x-y}{y}\Big|^{1/2} = 2|t|^{1/2}|\cdot |x-y|^{1/2}.
    \]
    A similar argument applies for the sine terms, which delivers
    \[
        |y|^{1/2}\cdot|x^{it}-y^{it}| \le 4|t|^{1/2}\cdot|x-y|^{1/2}
    \] The case $y > x > 0$ follows analogously by interchanging the roles of $x$ and $y$.
    Therefore,
    \[
        |g(x) - g(y)| = |x^{1/2 + it} - y^{1/2 + it}| \le (1+4|t|^{1/2})|x-y|^{1/2}.
    \]
\end{proof}

\begin{proof}[Proof of Lemma \reff{lem:transfer_dependency}]
    For a sequence $W_j := J_{j,n}(\sG_i)$ with $\Vert W_1 \Vert_{1} < \infty$, let $\sP_{j-k}(W) := \IE[W \mid \sG_{j-k}] - \IE[W \mid \sG_{j-k-1}]$ denote its projection, $k \in \IN_0$. Then by \cite{wu2005anotherlook} the projection property of the conditional expectation and an elementary property of $\delta_2$ (cf. \cite{wu2005anotherlook}, Theorem 1), we have
    \begin{eqnarray*}
        \Var(\sum_{j=1}^n g(X_j))^{1/2} &=& \sum_{k=0}^\infty \Big\Vert  \sum_{j = 1}^n\sP_{j-k}g(X_j) \Big\Vert_2 \\
        &\le& n^{1/2} \sum_{k=0}^\infty  \delta_2^{g(X)}(k).
    \end{eqnarray*}
    By Lemma \reff{lem:holder},
    \begin{eqnarray*}
        \delta_2^{g(X)}(k) &=& \sup_{j=1,...,n} \Vert g(X_j) - g(X_j)^{*(j-k)} \Vert_2 \\
        &\le& L(t) \cdot \sup_{j=1,...,n}\Vert X_j - X_j^{*(j-k)} \Vert_1^{1/2} = L(t) \cdot \delta_1^X(k)^{1/2},
    \end{eqnarray*} which delivers our desired statement.
\end{proof}
\begin{proof}[Proof of Theorem \ref{theorem:lower_bound}]
	We first outline the main steps of the proof.  We will construct a family of
	functions in $\IW^s_{1/2}(L)$ by a perturbation of the
	survival function $S_o: \IR_+ \rightarrow \IR_+$ with small ``bumps", such that their $\IL^2$-distance
	and the Kullback-Leibler divergence of their induced distributions can be bounded from below and
	above, respectively. The claim follows  by applying Theorem 2.5
	in \cite{Tsybakov2008}. We use
	the following construction, which we present first.\\ 
	Denote by $C_c^{\infty}(\IR)$  the set of all smooth functions with compact support in $\IR$ and let $\psi\in C_c^{\infty}(\IR)$ be a function with support in $[0,1]$ and $\int_0^{1} \psi(x)dx = 0$. For each $K\in\IN$ (to be selected below) and
	$k\in\{0, \dots, K\}$ we define the bump-functions
	$\psi_{k, K}(x):= \psi(xK-K-k),$ $x\in\IR$. For $j\in \IN_0$ we set the finite constant $C_{j,\infty}:= \max(\|\psi^{(l)}\|_{\infty}, l\in \{0,\dots,j\})$. Let us further define the operator $\mathcal S: C_c^{\infty}(\IR)\rightarrow C_c^{\infty}(\IR)$ with $\mathcal S[f](x)=-x f^{(1)}(x)$ for all $x\in \IR$ and define $\mathcal S^1:=\mathcal S$ and $\mathcal S^{n}:=\mathcal S \circ \mathcal S^{n-1}$ for $n\in \IN, n\geq 2$.  Now, for $j \in \IN$, we define the function $
	\psi_{k,K,j}(x):= \mathcal S^{j} [\psi_{k,K}](x)=(-1)^j\sum_{i=1}^{j} c_{i,j} x^i K^{i} \psi^{(i)}(xK-K-k)$ for $x \in \IR_+$ and $c_{i,j} \geq 1$. \\
	For a bump-amplitude $\delta>0, \gamma \in \IN$ and a vector 
	$\bm{\theta}=(\theta_1,\dots,\theta_K)\in \{0,1\}^K$ we define
	\begin{equation}\label{equation:lobosurv}
	S_{\bm{\theta}}(x)=S_o(x)+ \delta K^{-s-\gamma+1} \sum_{k=0}^{K-1}
	\theta_{k+1} \psi_{k, K,\gamma-1}(x)\text{ where } S_o(x):=\exp(-x).
	\end{equation}
	Taking the negative sign of the derivative of this function leads to the density
	\begin{equation}\label{equation:lobodens}
	f_{\bm{\theta}}(x)=f_o(x)+ \delta K^{-s-\gamma+1} \sum_{k=0}^{K-1}
	\theta_{k+1} \psi_{k, K,\gamma}(x)x^{-1}\text{ where } f_o(x):=\exp(-x).
	\end{equation}
	Until now, we did not give a sufficient condition to ensure that our constructed functions $\{S_{\bm{\theta}}: \bm{\theta} \in \{0,1\}^K\}$ are in fact survival functions. We do this by stating conditions such that the family $\{f_{\bm{\theta}}: \bm{\theta} \in \{0,1\}^K\}$ is a family of densities.
	\begin{lem}\label{mm:lem:den}
		Let $0<\delta< \delta_o(\psi, \gamma):=\exp(-2)2^{-\gamma} (C_{\gamma,\infty} c_{\gamma})^{-1}$. Then for all $\bm{\theta}\in\{0,1\}^K$, $f_{\bm{\theta}}$ is a density.
	\end{lem}
	Furthermore, it is possible to show that these survival functions all lie inside the ellipsoids $\IW^s_{1/2}(L)$ for $L$ big enough. This is captured in the following lemma.
	\begin{lem}\label{lemma:Lag_SobDen}Let
		$s\in\IN$. Then,
		there is $L_{s,\gamma,\delta}>0$ such that $S_o$
		and any $S_{\bm{\theta}}$ as in  \eqref{equation:lobodens} with
		$\bm\theta\in\{0,1\}^K$, $K\in\IN$,  belong to $\IW^s_{1/2}(L_{s,\gamma,\delta})$.
	\end{lem}
	For sake of simplicity we denote for a function $\varphi \in \IL^2(\IR_+,x^0)\cap \IL^1(\IR_+, x^{-1/2})$ the multiplicative convolution with $g$ by $\widetilde{\varphi} := \varphi *g$. Next, we see that for $y_2 \geq y_1 >0$,
	\begin{align*}
	\widetilde{f_o}(y_1) = \int_0^{\infty} g(x) x^{-1} \exp(-y_1/x) dx \geq \int_0^{\infty} g(x) x^{-1} \exp(-y_2/x) dx = \widetilde{f_o}(y_2)
	\end{align*} and thus $\widetilde{f_o}$ is monotonically decreasing. Additionally, we have that $\widetilde f_o(2)>0$, since otherwise $g=0$ almost everywhere.
	Exploiting \textit{Varshamov-Gilbert's
		lemma} (cf. \cite{Tsybakov2008}) in Lemma \ref{lemma:tsyb_vorb} we show
	that there exists $M\in\IN$ with $M\geq 2^{K/8}$ and a subset
	$\{\bm \theta^{(0)}, \dots, \bm \theta^{(M)}\}$ of $\{0,1\}^K$ with
	$\bm \theta^{(0)}=(0, \dots, 0)$ such that for all
	$j, l \in \{0, \dots, M\}$, $j \neq l$, the $\IL^2$-distance and the
	Kullback-Leibler divergence are bounded for $K\geq K_o(\gamma,\psi)$.
	\begin{lem}\label{lemma:tsyb_vorb}
		Let $K\geq \max\{K_o(\psi,\gamma),8\}$. Then there exists a subset $\{\bm \theta^{(0)}, \dots, \bm \theta^{(M)}\}$ of $\{0,1\}^K$  with $\bm \theta^{(0)}=(0, \dots, 0)$ such that $M\geq 2^{K/8}$ and for all $j, l \in \{ 0,\dots, M\} , j \neq l$, $\|S_{\bm\theta^{(j)}}-S_{\bm\theta^{(l)}}\|^2 \geq \frac{\|\psi^{(\gamma-1)}\|^2\delta^2}{16}  K^{-2s}$ and $ \text{KL}(\widetilde{f}_{\bm\theta^{(j)}}, \widetilde{f}_{\bm\theta^{(0)}}) \leq \frac{C_1(g)\|\psi\|^2}{\widetilde{f}_o(2)\log(2)} \delta^2 \log(M) K^{-2s-2\gamma+1}$ where $\text{KL}$ is the Kullback-Leibler-divergence.
	\end{lem}
	Selecting $K=\lceil n^{1/(2s+2\gamma-1)}\rceil$ delivers
	\begin{align*}
	\frac{1}{M}\sum_{j=1}^M
	\text{KL}((\widetilde{f}_{\bm{\theta^{(j)}}})^{\otimes
		n},(\widetilde{f}_{\bm{\theta^{(0)}}})^{\otimes n})
	&= \frac{n}{M} \sum_{j=1}^M \text{KL}(
	\widetilde{f}_{\bm{\theta^{(j)}}},\widetilde{f}_{\bm{\theta^{(0)}}} )
	\leq c_{\psi, \delta,g,\gamma,f_o}  \log(M)
	\end{align*}
	where $c_{\psi, \delta,g,\gamma,f_o}< 1/8$ for all  $\delta\leq \delta_1(\psi, g, \gamma, f_o)$ and $M\geq 2$ for
	$n\geq n_{s,\gamma}:=\max\{8^{2s+1}, K_o(\gamma, \psi)^{2s+2\gamma+1}\}$. Thereby, we can use Theorem 2.5 of
	\cite{Tsybakov2008}, which in turn for any estimator $S$ of $\IW^s_{1/2}(L)$
	implies
	\begin{align*}
	\sup_{S\in\IW^s_{1/2}(L)}
	\IP\big(\|\widehat S-S\|^2\geq
	\tfrac{c_{\psi, \delta, \gamma}}{2}n^{-2s/(2s+2\gamma-1)} \big)\geq
	\tfrac{\sqrt{M}}{1+\sqrt{M}}\big(1-1/4
	-\sqrt{\frac{1}{4\log(M)}} \big) \geq 0.07.
	\end{align*}
	Note that the constant $c_{\psi, \delta,\gamma}$ does only depend on
	$\psi,\gamma $ and $\delta$. Hence,
	it is independent of the parameters $s$ and $n$. The claim
	of Theorem \ref{theorem:lower_bound} follows by using Markov's inequality,
	which completes the proof.
\end{proof}

\paragraph{Proofs of the lemmata}

\begin{proof}[Proof of Lemma \ref{mm:lem:den}]
	For any $h\in C_c^{\infty}(\IR)$ we can state that $\mathcal S[h]\in C_c^{\infty}(\IR)$ and thus $\mathcal S^j[h]\in C_c^{\infty}(\IR)$ for any $j\in \IN$. Further, for $h\in C_c^{\infty}(\IR)$, $\int_{-\infty}^{\infty} h'(x)dx=0$, which implies that for any $\delta >0$ and $\bm{\theta}\in\{0,1\}^K$ we have $\int_0^{\infty} f_{\bm{\theta}}(x)dx= 1$.\\
	Now due to the construction \eqref{equation:lobodens} of the functions $\psi_{k,K}$ we easily see that the function  $\psi_{k,K}$ has support on $[1+k/K,1+(k+1)/K]$ which leads to  $\psi_{k,K}$ and $ \psi_{l,K}$ having disjoint supports if $k\neq l$. Here, we want to emphasize that $\mathrm{supp}(\mathcal S[h]) \subseteq \mathrm{supp}(h)$ for all $h\in C_c^{\infty}(\IR)$ which implies that $\psi_{k,K,\gamma}$ and $ \psi_{l,K,\gamma}$ have disjoint supports if $k\neq l$, too.
	For $x\in [1,2]^c$ we have $f_{\bm{\theta}}(x)=\exp(-x)\geq 0$. Now let us consider the case $x\in[1,2]$. In fact there is $k_o\in\{0,\dots,K-1\}$ such that $x \in [1+k_o/K,1+ (k_o+1)/K]$ and hence
	\begin{equation*}
	S_{\bm{\theta}}(x)= S_{o}(x) + \theta_{k_o+1}\delta K^{-s-\gamma+1} x^{-1}\psi_{k_o,K, \gamma}(x) \geq \exp(-2)  - \delta 2^{\gamma} C_{\gamma, \infty} c_{\gamma}
	\end{equation*}
	since $\|\psi_{k,K,j}\|_{\infty} \leq 2^{j} C_{j, \infty} c_j K^{j}$ for any $k\in \{0,\dots, K-1\}$, $s\geq1$ and $j\in \IN$ where $c_j:= \sum_{i=1}^j c_{i,j}$. Choosing $\delta\leq \delta_o(\psi,\gamma)=\exp(-2)2^{-\gamma} (C_{\gamma,\infty} c_{\gamma})^{-1}$ ensures $f_{\bm{\theta}}(x) \geq 0$ for all $x\in \IR_+.$ 
\end{proof}

\begin{proof}[Proof of Lemma \ref{lemma:Lag_SobDen}]Our proof starts with the
	observation that for all $t\in \IR$ we have $\sM_{1/2}[S_o](t)=\Gamma(1/2+it)$.
	Now, by applying the Stirling formula (cf. \cite{BelomestnyGoldenshluger2020}) we get $|\Gamma(1/2+it)| \sim  \exp(-\pi/2 |t|)$, $|t|\geq 2$. Thus for every $s\in \IN$ there exists $L_{s}$ such that $|S_o|_s^2 \leq L $ for all $L\geq L_s$. \\
	Next we consider $|S_o-S_{\bm{\theta}}|_s$. Let us therefore define first $\Psi_K:= \sum_{k=0}^{K-1} \theta_{k+1} \psi_{k,K}$ and $\Psi_{K,j}:= \mathcal S^j [\Psi_K]$ for $j\in \IN$. Then we have
	$|S_o-S_{\bm{\theta}}|_s^2= \delta^2 K^{-2s-2\gamma+2} |\Psi_{K,\gamma-1}|_s^2$ where $|\, . \,|_s$ is defined in \reff{eq:ani:mell:sob}. Now since for any $j\in \IN$, $\mathrm{supp}(\Psi_{K,j}) \subset [1,2], \|\Psi_{K,j}\|_{\infty} < \infty$ we have that the Mellin transform of $\Psi_{K,j}$ is defined for any $c\in (0,\infty)$. By a recursive application of the integration by parts we deduce that $|\sM_{1/2}[\Psi_{K,s+\gamma-1}](t)|^2 = (1/4+t^2)^s|\sM_{1/2}[\Psi_{k,\gamma-1}](t)|^2$, whence
	\begin{align*}
	|\Psi_{k,\gamma-1}|_s^2&\leq C_s\int_{-\infty}^{\infty} |\sM_{1/2}[\Psi_{K,s+\gamma-1}](t)|^2 dt =C_s\int_0^{\infty}  |\Psi_{K,s+\gamma-1}(x)|^2 dx
	\end{align*}
	by the Parseval formula, see equation \reff{eq:Mel:plan}, where $C_s>0$ is a positive constant. Since the $\psi_{k,K}$ have disjoint support for different values of $k$ we reason that $|\Psi_{k,\gamma-1}|_s^2 \leq C_s \sum_{k=0}^{K-1} \theta_{k+1}^2 \int_0^{\infty}  |\mathcal S^{\gamma-1+s}[\psi_{k,K}](x)|^2dx$.
	Applying Jensen's inequality and the fact that $\mathrm{supp}(\psi_{k,K})\subset [1,2]$, we obtain to
	\begin{align*}
	|\Psi_{k,\gamma-1}|_s^2&\leq C_{(\gamma,s)} \sum_{k=0}^{K-1} \sum_{j=1}^{\gamma+s-1} c_{j,\gamma-1+s}^2 \int_1^{2} x^{2j} K^{2j} \psi^{(j)}(xK-K-k)^2 dx\\
	&\leq  C_{(\gamma, s)} K^{2(\gamma-1+s)} \sum_{k=0}^{K-1} \sum_{j=1}^{\gamma+s-1} c_{j,\gamma+s}^2  4^j C_{\psi,s,\gamma}^2 K^{-1} \leq  C_{(\gamma, s)} K^{2(\gamma-1+s)}.
	\end{align*} 
	Thus, $|S_o-S_{\bm{\theta}}|_s^2 \leq C_{( s, \gamma,\delta)}$ and  $|S_{\bm{\theta}}|_s^2 \leq 2(|S_o-S_{\bm{\theta}}|_s^2 + |S_o|_s^2) \leq 2(C_{(s, \gamma,\delta)}+ L_s) =: L_{s, \gamma,\delta,1}$.
	By Corollary \ref{cor:upper:minimax} it is sufficient to show that $\int_0^{\infty} x f_{\bm \theta}(x)dx \leq L_{s, \gamma,\delta, 2}$. In fact,
	\begin{align*}
	   \int_{0}^{\infty} x f_{\bm \theta}(x)dx = 1 + \delta K^{-s-\gamma+1} \sum_{k=0}^{K-1} \int_{1+k/K}^{1+(k+1)/K} \psi_{k,K,\gamma}(x) dx \leq 1+\delta C_{\gamma}
	\end{align*}
	since $\|\psi_{k,K,\gamma}\|_{\infty} \leq 2^{\gamma} C_{\gamma, \infty} c_{\gamma} K^{\gamma}=C_{\gamma} K^{\gamma}$, cf. Proof of Lemma \ref{mm:lem:den}. The claim follows by choosing $L_{s,\gamma,\delta}=\max\{L_{s,\gamma,\delta,1},L_{s,\gamma,\delta,2}\}$.
\end{proof}
\begin{proof}[Proof of Lemma \ref{lemma:tsyb_vorb}]  $\ $\\
	Using the fact that the functions $(\psi_{k,K, \gamma})_{k\in\{0,\dots,K-1\}}$ with different index $k$ have disjoint supports we get
	\begin{align*}
	\| S_{\bm{\theta}}-S_{\bm{\theta}'}\|^2&= \delta^2 K^{-2(s+\gamma-1)} \| \sum_{k=0}^{K-1} ( \theta_{k+1}-\theta'_{k+1}) \psi_{k,K, \gamma-1}\|^2\\
	& = \delta^2  K^{-2(s+\gamma-1)}\rho(\bm \theta, \bm \theta')  \|\psi_{0,K, \gamma-1}\|^2
	\end{align*}
	with $\rho(\bm \theta, \bm \theta'):= \sum_{j=0}^{K-1} \Ii_{\{\bm \theta_{j+1} \neq \bm \theta'_{j+1}\}}$, the \textit{Hamming distance}. Now the first claim follows by showing that  $ \|\psi_{0,K, \gamma-1}\|^2 \geq \frac{K^{2\gamma-3}\|\psi^{(\gamma-1)}\|^2 }{2} $ for $K$ big enough.
	To do so, we observe that 
	\begin{align*} \| \psi_{0,K, \gamma-1} \|^2 = \sum_{i,j \in\{1,\dots,\gamma-1\}} c_{j,\gamma-1}c_{i,\gamma-1} \int_0^{\infty} x^{j+i+1} \psi_{0,K}^{(j)}(x)  \psi_{0,K}^{(i)}(x)dx. \end{align*}
	Defining $\Sigma:= \| \psi_{0,K,\gamma-1}\|^2 -  \int_0^{\infty}(x^{\gamma-1} \psi_{0,K}^{(\gamma-1)}(x))^2dx$ we see that 

	\begin{align}\label{equation:low_l2_gamma}
	\hspace*{-1cm}\| \psi_{0,K, \gamma-1} \|^2 &= \Sigma + \int_0^{\infty}(x^{\gamma-1} \psi_{0,K}^{(\gamma-1)}(x))^2 dx
	\geq \Sigma + K^{2\gamma-3} \| \psi^{(\gamma-1)} \|^2 \geq \frac{ K^{2\gamma-3} \| \psi^{(\gamma-1)} \|^2}{2}
	\end{align}
	as soon as $|\Sigma|\leq \frac{ K^{2\gamma-3} \| \psi^{(\gamma-1)} \|^2}{2}$. This is obviously true as soon as $K \geq K_{o}(\gamma, \psi)$ and thus $	\| S_{\bm{\theta}}-S_{\bm{\theta}'}\|^2 \geq  \frac{\delta^2\|\psi^{(\gamma-1)}\|^2 }{2}K^{-2s-1}\rho(\bm{\theta},\bm{\theta}')$ for $K\geq K_o(\psi, \gamma)$.\\
	Now we use the \textit{Varshamov-Gilbert Lemma} (cf. \cite{Tsybakov2008}) which states that for $K \geq 8$ there exists a subset $\{\bm \theta^{(0)}, \dots, \bm\theta^{(M)}\}$ of $\{0,1\}^K$ with $\bm \theta^{(0)}=(0, \dots, 0)$ such that $\rho(\bm \theta^{(j)}, \bm \theta^{(k)}) \geq K/8$ for all $j ,k\in \{ 0, \dots, M\}, j\neq k $ and $M\geq 2^{K/8}$. Therefore, $\| S_{\bm\theta^{(j)}}- S_{\bm\theta^{(l)}}\|_{\omega}^2 \geq \frac{\|\psi^{(\gamma-1)}\|^2\delta^2}{16}  K^{-2s}.
	$\\
	For the second part we have $f_o=f_{\bm\theta^{(0)}}$, and by using $\text{KL}(\widetilde{f}_{\bm \theta}, \widetilde{f}_o) \leq \chi^2(\widetilde{f}_{\bm\theta},\widetilde{f}_o):= \int_{\IR_+} (\widetilde{f}_{\bm\theta}(x) -\widetilde{f}_o(x))^2/\widetilde{f}_o(x) dx$ it is sufficient to bound the $\chi^2$-divergence. We notice that $\widetilde{f}_{\bm\theta} -\widetilde{f}_o$ has support on $[0,2]$ since $f_{\bm{\theta}}-f_o$ has support on $[1,2]$ and $g$ has support on $[0,1]$. In fact for $y>2$, $\widetilde{f}_{\bm\theta}(y) -\widetilde{f}_o(y)=\int_y^{\infty} (f_{\bm{\theta}}-f_o)(x)x^{-1} g(y/x)dx =0$. Let $\Psi_{K,\gamma}:= \sum_{k=0}^{K-1} \theta_{k+1} \psi_{k,K, \gamma}= \mathcal S^{\gamma}[\sum_{k=0}^{K-1} \theta_{k+1} \psi_{k,K}]=: \mathcal S^{\gamma}[\Psi_K]$. Now by using the  compact support property and a single substitution we get
	\begin{align*}
	\chi^2(\widetilde{f}_{\bm\theta},\widetilde{f}_o)
	\leq \widetilde{f}_o(2)^{-1} \|\widetilde{f}_{\bm \theta}- \widetilde{f}_o\|^2 &=\widetilde{f}_o(2)^{-1}  \delta^2 K^{-2s-2\gamma+2} \| \widetilde{\omega^{-1}\Psi}_{K,\gamma}\|^2.
	\end{align*}
	Let us now consider $\| \widetilde{\omega^{-1}\Psi}_{K,\gamma}\|^2$. In the first step we see by application of the Parseval equality that $\|\widetilde{\omega^{-1}\Psi}_{K,\gamma}\|^2 = \frac{1}{2\pi} \int_{-\infty}^{\infty} |\sM_{1/2}[\widetilde{\omega^{-1}\Psi}_{K,y}](t)|^2dt$. Now for  $t\in \IR$, we see by using the multiplication theorem for Mellin transforms  that $\sM_{1/2}[\widetilde{\omega^{-1}\Psi}_{K,\gamma}](t)= \sM_{1/2}[ g](t) \cdot \sM_{1/2}[\omega^{-1}\mathcal S^{\gamma}[\Psi_{K}]](t)$.  Again we have $\sM_{1/2}[ \mathcal \omega^{-1}S^{\gamma}[\Psi_{K}]](t)=(-1/2+it)^{\gamma} \sM_{-1/2}[\Psi_{K}](t)$. Together with assumption \textbf{[G1']} we obtain
	\begin{align*}
	\|\widetilde{\omega^{-1}\Psi}_{K,\gamma}\|^2  \leq \frac{C_1(g)}{2\pi}\int_{-\infty}^{\infty} |\sM_{-1/2}[ \Psi_{K}](t)|^2 dt= C_1(g) \| \omega^{-1}\Psi_K\|^2  \leq C_1(g) \|\psi\|^2. 
	\end{align*}
	Since $M \geq 2^K$ we have
	$\text{KL}(\widetilde{f}_{\bm\theta^{(j)}},\widetilde{f}_{\bm\theta^{(0)}})\leq \frac{C_1(g)\|\psi\|^2}{\widetilde f_o(2)\log(2)} \delta^2 \log(M) K^{-2s-2\gamma+1}.$
\end{proof}

\subsection{Proof of Section \ref{sec_datadriven}}

\begin{proof}[Proof of Theorem \ref{dd:thm:ada}] 
	Let us define the nested subspaces $(U_k)_{k\in \IR_+}$ by $U_k:=\{h\in \IL^2(\IR_+,x^0): \forall |t|\geq k: \sM_{1/2}[h](t)=0\}$. For any $h \in U_k$ we consider the empirical contrast 
	\begin{align*}
	\gamma_n(h) = \|h\|^2 -2  \frac{1}{2\pi} \int_{-\infty}^{\infty} \widehat{\sM}(t)\frac{\sM_{1/2}[h](-t)}{(1/2+it)\sM_{3/2}[g](t)}dt= \|h\|^2 - 2 n^{-1} \sum_{j=1}^n \nu_h(Y_j) 
	\end{align*}
	 with $\nu_h(Y_j):= \frac{1}{2\pi} \int_{-\infty}^{\infty} Y_j^{1/2+it} \frac{\sM_{1/2}[h](-t)}{(1/2+it)\sM_{3/2}[g](t)}dt$. It can easily be seen that $\widehat S_k = \argmin\{\gamma_n(h):h\in U_k\}$ with $\gamma_n(\widehat S_k)=-\|\widehat S_k\|^2$.
	For $h\in U_k$  define the centered empirical process $\bar{\nu}_h:= n^{-1} \sum_{j=1}^n \nu_h(Y_j) - \langle h, S \rangle$. Then we have for $h_1,h_2 \in U_k$,
	\begin{align}\label{dd:eq:cont}
	\gamma_n(h_1)-\gamma_n(h_2) = \|h_1-S\|^2 - \|h_2-S\|^2 - 2 \bar{\nu}_{h_1-h_2}.
	\end{align}
	Now since $\gamma_n(\widehat S_k) \leq \gamma_n(S_k)$ and by the definition of $\widehat k$ we have $\gamma_n(\widehat S_{\widehat k})- \widehat{\mathrm{pen}}(\widehat k) \leq \gamma_n(\widehat S_{ k})- \widehat{\mathrm{pen}}( k) \leq \gamma_n(S_k)- \widehat{\mathrm{pen}}(k)$ for any $k\in \mathcal K_n$. Now using \eqref{dd:eq:cont}, 
	\begin{align*}
	\| S- \widehat S_{\widehat k} \|^2 \leq \|S-S_k\|^2 + 2 \bar\nu_{\widehat S_{\widehat k}- S_k} + \widehat{\mathrm{pen}}( k)- \widehat{\mathrm{pen}}(\widehat k).
	\end{align*}
	First we note that $U_{k_1} \subseteq U_{k_2}$ for $k_1 \leq k_2$. Let us now denote by $a\vee b:= \max\{a,b\}$ and define for all $k\in \mathcal K_n$ the unit balls $B_k:=\{h\in U_k: \|h\| \leq 1\}$. Next we deduce from $2ab \leq a^2+b^2$ that $ 2\bar\nu_{\widehat S_{\widehat k}-S_k}\leq 4^{-1}\|\widehat  S_{\widehat k} - S_k\|^2 + 4 \sup_{h \in B_{\widehat k \vee k} }\bar \nu_h^2$. Furthermore, we see that $4^{-1} \|\widehat  S_{\widehat k} - S_k\|^2  \leq 2^{-1} (\|\widehat  S_{\widehat k} - S\|^2 + \|S- S_k\|^2)$. Putting all the facts together and defining 
	\begin{align}\label{eq:delta}
	\hspace*{-0.8cm}p(\widehat k \vee k):= 24 \sigma_{Y}\Delta_g(\widehat k \vee k)n^{-1} 
	\end{align} 
	 we have
	\begin{align*}
	\hspace*{-1cm}	\| S- \widehat S_{\widehat k} \|^2 \leq &3\|S-S_k\|^2 + 8 \big(\sup_{h \in B_{\widehat k \vee k} }\bar\nu_{h}^2 -p(k\vee \widehat k) \big)_+
	+8p(\widehat k \vee k) + 2\widehat{\mathrm{pen}}(k)- 2\widehat{\mathrm{pen}}(\widehat k)
	\end{align*}
	The decomposition $\bar\nu_h =\bar\nu_{h,in}+\bar\nu_{h,de}$ where
	\begin{align*}
	    \bar\nu_{h,in}:= n^{-1} \sum_{j=1}^n (\nu_h(Y_j)- \IE_{|X}(\nu_h(Y_j))) \text{ and } \bar\nu_{h, de}=n^{-1} \sum_{j=1}^n \IE_{|X}(\nu_h(Y_j))) - \IE(\nu_h(Y_j))
	\end{align*}
    implies the inequality
	\begin{align*}
	\hspace*{-1cm}	\| S- \widehat S_{\widehat k} \|^2 \leq &3\|S-S_k\|^2 + 16 \big(\sup_{h \in B_{\widehat k \vee k} }\bar\nu_{h,in}^2 -\frac{1}{2}p(k\vee \widehat k) \big)_+\\
	&+16\sup_{h \in B_{\widehat k \vee k} }\bar\nu_{h,de}^2+8p(\widehat k \vee k)+ 2\widehat{\mathrm{pen}}(k)- 2\widehat{\mathrm{pen}}(\widehat k).
	\end{align*}
	Assuming that $\chi \geq 96$, $4p(\widehat k \vee k) \leq \mathrm{pen}(k)+ \mathrm{pen}(\widehat k)$. Thus,
	\begin{align*}
	\| S- \widehat S_{\widehat k} \|^2 \leq &6\big(\|S-S_k\|^2 +\mathrm{pen}(k) \big) +16 \max_{k\in \mathcal K_n}\big(\sup_{h \in B_{k} }\bar\nu_{h,in}^2 -\frac{1}{2}p( k) \big)_+\\
	&  + 16\max_{k'\in \mathcal K_n} \sup_{h\in B_{k'}} \bar{\nu}_{h,de}^2 + 2(\widehat{\mathrm{pen}}(k)- 2\mathrm{pen}(k))+ 2(\mathrm{pen}(\widehat k)-\widehat{ \mathrm{pen}}(\widehat k))_+.
	\end{align*}
	We will use the following Lemmata which will be proven afterwards.
		\begin{lem}\label{dd:lem:1}
		Under the assumption of Theorem \ref{dd:thm:ada} we have 
		\begin{align*}
		    \IE(\max_{k\in \mathcal K_n} \sup_{h\in B_k} \bar{\nu}_{h,de}^2) \leq \frac{1}{2\pi}\int_{-n}^n \Var(\sM_{1/2}[\widehat S_X](t))dt
		\end{align*}
	\end{lem}
	\begin{lem}\label{dd:lem:2}
		Under the assumption of Theorem \ref{dd:thm:ada} we have 
	\begin{align*}
	\IE\big(\max_{k\in \mathcal K_n} (\sup_{h \in B_{k} }\bar \nu_{h,in}^2 -\frac{1}{2}p(k) \big)_+  \leq   C(g)\big(\frac{\sigma_X}{n} + \frac{\IE(X_1^{5/2})}{\sigma_X^{3/2} n} + \frac{\Var(\widehat\sigma_X)}{\sigma_X}\big) 
	\end{align*}
	\end{lem}
	\begin{lem}\label{dd:lem:3}
	Under the assumption of Theorem \ref{dd:thm:ada} we have 
	\begin{align*}
	    \IE((\mathrm{pen}(\widehat k)-\widehat{\mathrm{pen}}(\widehat k) )_+)\leq  4\chi \frac{\IE(Y_1^2)}{\sigma_Y n} + 4\chi \frac{\sigma_U}{\sigma_X}\Var(\widehat{\sigma}_X).
	\end{align*}
	\end{lem}
	Applying the lemmata and using the fact that $\IE(\widehat{\mathrm{pen}}(k))=2\mathrm{pen}(k)$,
	\begin{align*}
	\IE (	\| S- \widehat S_{\widehat k} \|^2) \leq 6\big(\|S-S_k\|^2 +\mathrm{pen}(k) \big) + \frac{C(g,\chi,f)}{n} +C(g,\chi)\frac{\Var(\widehat \sigma_X)}{\sigma_X} + \frac{1}{2\pi}\int_{-n}^n \Var(\sM_{1/2}[\widehat S_X](t))dt.
	\end{align*}
	Since this inequality holds for all $k\in \mathcal K_n$, it implies the claim.
\end{proof}
	 \paragraph{Proof of the Lemmata}
\begin{proof}[Proof of Lemma \ref{dd:lem:1}]
	By applying the Cauchy-Schwarz inequality we get for any $k\in \mathcal K_n$ and any $h\in B_k$,
	\begin{align*}
	    n^{-1} \sum_{j=1}^n \IE_{|X}(\nu_h(Y_j))-\IE(\nu_h(Y_j))&= \frac{1}{2\pi} \int_{-k}^k \left(n^{-1} \sum_{j=1}^n \IE_{|X}(Y_j^{1/2+it})- \IE(Y_j^{1/2+it})\right) \frac{\sM_{1/2}[h](-t)}{(1/2+it)\sM_{3/2}[g](t)} dt\\
	    &= \frac{1}{2\pi} \int_{-k}^k \left(n^{-1} \sum_{j=1}^n X_j^{1/2+it}- \IE(X_j^{1/2+it})\right) \frac{\sM_{1/2}[h](-t)}{(1/2+it)} dt\\
	    &\leq  \left(\frac{1}{2\pi}\int_{-k}^k \frac{|n^{-1} \sum_{j=1}^n X_j^{1/2+it}- \IE(X_j^{1/2+it})|^2}{1/4+t^2} dt\right)^{1/2} \|h\|_{\IL^2(\IR_+,x^0)}.
	\end{align*}
	Now since $\|h\|_{\IL^2(\IR_+,x^0)}\leq 1,$ $\IE( \max_{k\in \mathcal K_n}\sup_{h\in B_k} \bar{\nu}_{h,de}^2) \leq  \frac{1}{2\pi}\int_{-n}^{n} \Var(\sM_{1/2}[\widehat S_X](t))dt$. 
	\end{proof}
	\begin{proof}[Proof of Lemma \ref{dd:lem:2}]
	First let us define $\widetilde p(k):= 24\sigma_U\widehat\sigma_X \Delta_g(k)n^{-1}$. Then, 
	\begin{align*}
    \max_{k\in \mathcal K_n}(\sup_{h\in B_k} \overline\nu_{h, in}^2-\frac{1}{2}p(k))_+ =\max_{k\in \mathcal K_n}(\sup_{h\in B_k} \overline\nu_{h, in}^2-\frac{1}{2}\widetilde p(k))_+ +\frac{1}{2}\max_{k\in \mathcal K_n} (\widetilde p(k)-p(k))_+.
\end{align*}

For the second summand we have
\begin{align*}
    \IE(\max_{k\in \mathcal K_n} (\widetilde p(k)-p(k))_+) \leq 24\sigma_U \IE((\widehat \sigma_X-\sigma_X)_+).
\end{align*}
Let us define $\Omega_X:=\{|\widehat\sigma_X-\sigma_X|\leq \sigma_X/2\}$. Then on $\Omega_X$ we have $\widehat\sigma_X \leq 3\sigma_X/2$ and thus $\IE((\widehat\sigma_X-\sigma_X)_+) =\IE((\widehat\sigma_X-\sigma_X)_+\Ii_{\Omega_X^c}) \leq 2 \sigma_X^{-1}\Var(\widehat\sigma_X)  $ by application of Cauchy-Schwarz and the Markov's inequality.

For the first summand we see
	\begin{align*}
	   \IE(\max_{k\in \mathcal K_n}(\sup_{h\in B_k} \bar{\nu}_{h,in}^2-\frac{1}{2}\widetilde p(k))_+) =\IE(\IE_{|X}(\max_{k\in \mathcal K_n}(\sup_{h\in B_k} \bar{\nu}_{h,in}^2-\frac{1}{2}\widetilde p(k))_+)).
	\end{align*}
	Thus we start by considering the inner conditional expectation to bound the term. By the construction of $\bar\nu_{h,in}$, its summands conditioned on $\sigma(X_i,i\geq 0)$ are independent but not identically distributed. We therefore split the process again in the following way
	\begin{align*}\bar\nu_{h,1}:= &n^{-1} \sum_{j=1}^n \nu_h(Y_j) \Ii_{(0,c_n)}(Y_j^{1/2})-\IE_{|X}(\nu_h(Y_1) \Ii_{(0,c_n)}(Y_1^{1/2})) \\
	&\text{ and }\bar\nu_{h,2}:=  n^{-1} \sum_{j=1}^n \nu_h(Y_j) \Ii_{(c_n,\infty)}(Y_j^{1/2})-\IE_{|X}(\nu_h(Y_1) \Ii_{(c_n,\infty)}(Y_1^{1/2}))\end{align*}
	to get
	\begin{align*}
	    \IE_{|X}(\max_{k\in \mathcal K_n}(\sup_{h\in B_k} |\bar{\nu}_{h,in}|^2-\frac{1}{2}\widetilde p(k))_+)&\leq 2\IE_{|X}(\max_{k\in \mathcal K_n} (\sup_{h\in B_k} |\bar{\nu}_{h,1}|^2- \frac{1}{4}\widetilde p(k))_+) +2\IE_{|X}(\max_{k\in \mathcal K_n}\sup_{h\in B_k} |\bar\nu_{h, 2}|^2), \\
	    &:= M_1+M_2
	\end{align*}
	where we will now consider the two summands $M_1, M_2$ separately.\\
	To bound the $M_1$ term we will use the Talagrand inequality \reff{tal:re3} 
	on the term $\IE_{|X}(\sup_{t\in B_k} |\bar{\nu}_{h,1}|^2-\frac{1}{4}\widetilde p(k))_+$. Indeed, we have
	\begin{align*}
	   M_1 \leq \sum_{k=1}^{K_n}\IE_{|X}(\sup_{t\in B_k} |\bar{\nu}_{h,1}|^2-\frac{1}{4}\widetilde p(k))_+,
	\end{align*}
	which will be used to show the claim. We want to emphasize that we are able to apply the Talagrand inequality on the sets $B_k$ since $B_k$ has a dense countable subset and due to continuity arguments.
	Further, we see that the random variables $\nu_h(Y_j)\Ii_{(0,c_n)}(Y_j^{1/2})- \IE_{|X}(\nu_h(Y_1)\Ii_{(0,c_n)}(Y_1^{1/2}))$, $j=1, \dots, n$, are conditioned on $\sigma(X_i,i\geq 0)$, centered and independent but not identically distributed. 
	In order to apply Talagrand's inequality, we need to find the constants $\Psi, \psi, \tau$ such that
	\begin{align*}
	    \sup_{h\in B_k} \sup_{y>0} |\nu_h(y)\Ii_{(0,c_n)}(y^{1/2})|\leq \psi; \quad \IE_{|X}(\sup_{h\in B_k} |\bar{\nu}_{h,1}|)\leq \Psi; \quad \sup_{h\in B_k} \frac{1}{n} \sum_{j=1}^n \Var_{|X}(\nu_h(Y_j)\Ii_{(0,c_n)}(Y_j^{1/2})) \leq \tau.
	\end{align*}
	We start to determine the constant $\Psi^2$. Let us define $\widetilde M(t):= n^{-1} \sum_{j=1}^n Y_j^{1/2+it}\Ii_{(0,c_n)}(Y_j^{1/2}))$ as an unbiased estimator of $\sM_{3/2}[f_Y\Ii_{(0,c_n)}](t)$ and 
	\begin{align*}
	\widetilde S_k(x):= \frac{1}{2\pi} \int_{-k}^k x^{-1/2-it}\frac{\widetilde {\mathcal M}(t)}{(1/2+it)\sM_{3/2}[g](t)}dt 
	\end{align*}
	where $n^{-1} \sum_{j=1}^n \nu_h(Y_j) \Ii_{(0,c_n)}(Y_i)=\langle \widetilde S_k, h \rangle.$
	Thus, we have for any $h\in B_{k}$ that $\bar \nu_{h,1}^2 = \langle h, \widetilde S_{k} - \IE_{|X}(\widetilde S_{k}) \rangle^2\leq \| h\|^2 \|\widetilde S_{k}- \IE_{|X}(\widetilde S_{k})\|^2$. Since $\|h\|\leq1$, we get 
	\begin{align*}
	\IE_{|X}( \sup_{h\in B_{k}} \bar\nu_{h,1}^2 ) \leq \IE_{|X}(\|\widetilde S_{k} - \IE_{|X}(\widetilde S_{k})\|^2) =\frac{1}{2\pi} \int_{-k}^k \frac{\IE_{|X}(|\widetilde \sM(t)- \IE_{|X}(\widetilde \sM(t))|^2)}{(1/4+t^2)|\sM_{3/2}[g](t)|^2} dt.
	\end{align*}
	Now since $Y_j^{1/2+it}\Ii_{(0,c_n)}(Y_j^{1/2})-\IE_{|X}(Y_j^{1/2+it}\Ii_{(0,c_n)}(Y_j^{1/2})$ are independent conditioned on $\sigma(X_i:i\geq 0)$ we obtain
	\begin{align*}
	    \IE_{|X}(|\widetilde \sM(t)- \IE_{|X}(\widetilde \sM(t))|^2) \leq  \frac{1}{n^2} \sum_{j=1}^n \IE_{|X}( Y_j \Ii_{(0, c_n)}(Y_j^{1/2}))=\frac{\sigma_U}{n} \widehat\sigma_X,
	\end{align*}
    which implies
	\begin{align*}
	    \IE( \sup_{h\in B_{k}} \bar\nu_{h,1}^2 ) \leq \sigma_U\widehat\sigma_X \frac{\Delta_g(k)}{n}=:\Psi^2.
	\end{align*}
	Thus $6\Psi^2 = \frac{1}{4}\widetilde p(k)$.\\
	Next we consider $\psi$. Let $y>0$ and $h\in B_{k}$. Then using the Cauchy-Schwarz inequality, $|\nu_h(y)\Ii_{(0,c_n)}(y)|^2= (2\pi)^{-2} c_n^2|\int_{-k}^{k} y^{it} \frac{\sM_{1/2}[h](-t)}{(1/2+it)\sM_{3/2}[g](t)} dt |^2 \leq (2\pi)^{-1} c_n^2\int_{-k}^{k}  |(1/2+it)\sM_{3/2}[g](t)|^{-2}dt\leq  c_n^2\Delta_g(k)=: \psi^2$ since $|y^{it}|=1$ for all $t\in \IR$.\\
	Next we consider $\tau$. In fact for $h\in B_{k}$ we can conclude
	\begin{align*}
	    \Var_{|X}&(\nu_h(Y_j)\Ii_{(0,c_n)}(Y_j^{1/2})) \leq \IE_{|X}(\nu_h(Y_j)^2) \\
	    &=\frac{1}{4\pi^2} \int_{-k}^{k}\int_{-k}^k \IE_{|X}(Y_j^{1+i(t_1-t_2)}) \frac{\sM_{1/2}[h](-t_1)}{(1/2+it_1) \sM_{3/2}[g](t_1)}\frac{ \sM_{1/2}[h](t_2)}{(1/2-it_2)\sM_{3/2}[g](-t_2)} dt_1 dt_2
	    \\
	    &=\frac{1}{4\pi^2} \int_{-k}^{k}\int_{-k}^k X_j^{1+i(t_1-t_2)} \IE_g(U_1^{1+i(t_1-t_2)})
	    \frac{\sM_{1/2}[h](-t_1)}{(1/2+it_1) \sM_{3/2}[g](t_1)}\frac{ \sM_{1/2}[h](t_2)}{(1/2-it_2)\sM_{3/2}[g](-t_2)} dt_1 dt_2\\
	    &=  X_j \int_0^{\infty} g(u)u \left|\sM_{1/2}^{-1}[\Ii_{[-k,k]}(t) \frac{\sM_{1/2}[h](t)}{(1/2-it) \sM_{3/2}[g](-t)}](u)\right|^2du.
	    \end{align*}
Taking the supremum of $u\mapsto ug(u)$ and applying Plancherel's theorem delivers
\begin{align*}
     \Var_{|X}(\nu_h(Y_j)\Ii_{(0,c_n)}(Y_j^{1/2})) \leq X_j \|xg\|_{\infty} \frac{1}{2\pi} \int_{-k}^k  \frac{|\sM_{1/2}[h](t)|^2}{|(1/2+it)\sM_{3/2}[g](t)|^2}dt.
\end{align*}
	Now since $\|h\|^2\leq 1$, and for $G_k(t):= \Ii_{[-k,k]}(t) |(1/2+it)\sM_{3/2}[g](t)|^{-2}$,
	\begin{align*}
	\sup_{h\in B_k} \frac{1}{n} \sum_{j=1}^n \Var_{|X}(\nu_h(Y_j)\Ii_{(0,c_n)}(Y_j^{1/2})) \leq \widehat \sigma_X \|G_k\|_{\infty} \|xg\|_{\infty} =:\tau.
	\end{align*}
	Hence, we have $\frac{n\Psi^2}{6\tau}= \frac{\sigma_U\Delta_g(k)}{6 \|xg\|_{\infty} \|G_k\|_{\infty} }$ and  $\frac{n\Psi}{100\psi}=\frac{\sqrt{\sigma_U\widehat\sigma_Xn}}{100c_n}$. Now, choosing $c_n:=\frac{\sqrt{\sigma_U\widehat\sigma_X n}}{a 100 \log(n)}$ gives $\frac{n\Psi}{100\psi}=a\log(n)$, and we deduce
	\begin{align*}
	\hspace*{-1cm}\IE_{|X}\big(\sup_{h \in B_{ k} }\bar \nu_{h,1}^2 -\frac{1}{4}\widetilde p(k) \big)_+  \leq \frac{C}{n} &\left( \widehat\sigma_X\|G_k\|_{\infty}\|xg\|_{\infty} \exp(-\frac{\pi\sigma_U\Delta_g(k)}{3\|xg\|_{\infty} \|G_k\|_{\infty} })\right.\\ 
	&\left.+  \frac{\sigma_U\widehat\sigma_X\Delta_g(k)}{\log(n)^2} n^{-a}\right).
	\end{align*}
	 Under \textbf{[G1]} we have $C_gk^{2\gamma-1}\geq \Delta_g(k) \geq c_g k^{2\gamma-1}$ and for all $t\in \IR$ it holds true that $c_g k^{2\gamma-2}\leq|G_k(t)| \leq C_g k^{2\gamma-2}$. Hence,
	\begin{align*}
	  \sum_{k=1}^{K_n} \IE_{|X}\big(\sup_{h \in B_{ k} }\bar \nu_{h,1}^2 -\frac{1}{4}\widetilde p(k) \big)_+  \leq \frac{C\widehat \sigma_X}{n} \left(\sum_{k=1}^{K_n} C(g) k^{2\gamma-1} \exp(-C(g)k) + \sum_{k=1}^{K_n} C(g)\frac{k^{2\gamma-1}}{\log(n)^2 n^{a}}\right)
	\end{align*}
	where the first sum is bounded in $n\in \IN$. The second sum can be bounded by $C(g)n^{\frac{2\gamma}{2\gamma-1}-a}/\log(n)^2$ which by choosing $a=\frac{2\gamma}{2\gamma-1}$ ensures the boundedness in $n\in\IN$. Thus we have
	\begin{align*}
	   \sum_{k=1}^{K_n} \IE_{|X}\big(\sup_{h \in B_{ k} }\bar \nu_{h,1}^2 -\frac{1}{4}\widetilde p(k) \big)_+ \leq \frac{C(g)\widehat\sigma_X}{n}.
	\end{align*}
	Now, we consider $M_2$. Let us define $\overline S_k:= \widehat S_k-\widetilde S_k$. Then from $\nu_{h,2}= \nu_{h,in}- \nu_{h,1}$ we deduce $\nu_{h,2}^2 =\langle \overline{S}_k-\IE_{|X}(\overline S_k), h\rangle^2\leq \| \overline{S_k}- \IE_{|X}(\overline S_k) \|^2$ for any $h\in B_k$. Further, for any $k\in\mathcal K_n$, $\| \overline{S_k}- \IE_{|X}(\overline S_k) \|^2 \leq \| \overline{S}_{K_n}- \IE_{|X}(\overline S_{K_n}) \|^2$ and
	\begin{align*}
	\IE_{|X}(\| \overline{S}_{K_n}- \IE_{|X}(\overline S_{K_n})) \|^2) &= \frac{1}{2\pi} \int_{-K_n}^{K_n} \Var_{|X}(\widehat{\mathcal M}(t)- \widetilde{\mathcal M}(t)) |(1/2+it) \sM_{3/2}[g](t)|^{-2} dt \\
	&\leq\frac{1}{ n^2} \sum_{j=1}^n \IE_{|X}(Y_j \Ii_{(c_n, \infty)}(Y_j^{1/2})) \Delta_g(K_n). 
	\end{align*}
	Let us define the event $\Xi_X:=\{\widehat\sigma_X\geq \sigma_X/2\}$. Then, we have
	\begin{align*}
	    \frac{1}{ n^2} \sum_{j=1}^n \IE_{|X}(Y_j \Ii_{(c_n, \infty)}(Y_j^{1/2})) \Delta_g(K_n) \Ii_{\Xi_X} \leq  \frac{1}{ n} \sum_{j=1}^n X_j^{1+p/2} \IE(U_j^{1+p/2}) c_n^{-p} \Ii_{\Xi_X}
	\end{align*}
	where on $\Xi_X$ we can state that $c_n^{-p} = C(g)n^{-p/2} (\widehat\sigma_X)^{-p/2} \log(n)^{p} \leq C(g)\sigma_X^{-p/2} n^{-p/2}\log(n)^p $.
	Choosing $p=3$ leads to $\IE_{|X}(\| \overline{S}_{K_n}- \IE_{|X}(\overline S_{K_n})) \|^2) \Ii_{\Xi_X} \leq \frac{C(g) \sigma_X^{-3/2}}{n} \IE(U_1^{5/2}) n^{-1}\sum_{j=1}^n X_j^{5/2}$. On the other hand,
	\begin{align*}
	    \frac{1}{ n^2} \sum_{j=1}^n \IE_{|X}(Y_j \Ii_{(c_n, \infty)}(Y_j^{1/2})) \Delta_g(K_n)  \Ii_{\Xi_X^c} \leq \frac{\sigma_X}{2} \Ii_{\Xi_X^c}\leq \frac{\sigma_X}{2} \Ii_{\Omega_X^c} .
	\end{align*}These three bounds imply
	\begin{align*}
	     \IE(\max_{k\in \mathcal K_n}(\sup_{h\in B_k} \overline\nu_{h, in}^2-\frac{1}{2}p(k))_+)\leq C(g)\big(\frac{\sigma_X}{2n} + \frac{\IE(X_1^{5/2})}{\sigma_X^{3/2} n} + \frac{2\Var(\widehat\sigma_X)}{\sigma_X}\big). 
	\end{align*}
\end{proof}
\begin{proof}[Proof of Lemma \ref{dd:lem:3}]
First we see that $\IE((\mathrm{pen}(\widehat k)-\widehat{\mathrm{pen}}(\widehat k))_{+})= 2\chi \IE((\sigma_Y/2-\widehat \sigma_Y)_+ \Delta_g(\widehat{k}) n^{-1}) \leq 2\chi \IE((\sigma_Y/2-\widehat \sigma_Y )_+)$. On $\Omega_Y:=\{|\sigma_Y-\widehat \sigma_Y|\leq \sigma_Y/2\}$ we have $\sigma_Y/2-\widehat \sigma_Y \leq 0$. Therefore,
	\begin{align*}
	\IE((\mathrm{pen}(\widehat k)-\widehat{\mathrm{pen}}(\widehat k))_{+}) \leq 2 \chi \IE((\sigma_Y/2- \widehat \sigma_Y)_+ \Ii_{\Omega^c}) \leq 2 \chi \sqrt{\Var_{f_Y}^n(\widehat\sigma_Y) \IP_{f_Y}(\Omega^c)}
	\end{align*}
	by applying the Cauchy-Schwarz inequality. Next, by Markov's inequality, $\IP[|\widehat \sigma_Y- \sigma_Y| \geq \sigma_Y/2] \leq 4\Var(\widehat \sigma_Y) \sigma_Y^{-2}$
	which implies $\IE((\mathrm{pen}(\widehat k)-\widehat{\mathrm{pen}}(\widehat k))_{+}) \leq 4\chi \Var(\widehat \sigma_Y) \sigma_Y^{-1} $. In analogy to the proof of Theorem \ref{thm:upp_bound} we get
	\begin{align*}
	     \Var(\widehat \sigma_Y) \leq \frac{\IE(Y_1^2)}{n} + \IE(U_1)^2\Var(\widehat{\sigma}_X).
	\end{align*}
\end{proof}

\begin{proof}[Proof of Corollary \ref{thm:upper_bound_adap}]
    We discuss each case separately. We already assessed the variance term in the integral in Corollary \ref{cor:consis_dep}. It remains to upper bound the variance of $\widehat\sigma_X$.
    \begin{enumerate}
        \item[\textbf{(I)}] Trivially, $\Var(n^{-1}\sum_{j=1}^n X_j) \le n^{-1}\IE(X_1^2)$.
        \item[\textbf{(B)}] Exploiting Lemma \ref{lem:beta_mix_var_bound} for the identity mapping $h = \mathrm{id}$, we have
        \[
            \Var(n^{-1}\sum_{j=1}^n X_j) \le n^{-2} \cdot 4n \IE(X_1^2b(X_1)) =  4 n^{-1} \IE(X_1^2b(X_1)).
        \]
        \item[\textbf{(F)}] Setting the function $g$ in Lemma \ref{lem:transfer_dependency} to be the identity mapping $g = \mathrm{id}$, we simply have
        \[
            \Var(\sum_{j=1}^n X_j)^{1/2} \le n^{1/2} \sum_{k=0}^\infty  \delta_2^{X}(k).
    \]
    \end{enumerate}
    Combined with the results of Corollary \ref{cor:consis_dep}, we derive our statement.
\end{proof}

\bibliographystyle{plain}
\bibliography{reference}

\end{document}